\documentclass[
a4paper,
reqno,
11pt,
]{article}
\usepackage{graphicx}
\usepackage{
amssymb,
amsmath,
amsthm,
eucal,
dsfont,
tikz}
\usepackage[margin=2cm]{caption}
\makeatletter
 
  \@addtoreset{equation}{section}
 \makeatother
\theoremstyle{plain}

\newtheorem{theorem}{Theorem}[section]
\newtheorem{lemma}[theorem]{Lemma}
\newtheorem{proposition}[theorem]{Proposition}
\newtheorem{corollary}[theorem]{Corollary}
\theoremstyle{definition}
\newtheorem{definition}[theorem]{Definition} 
\newtheorem{remark}[theorem]{Remark}

\newcommand{\bignorm}[1]{{\Big|\Big|#1\Big|\Big|}}
\newcommand{\norm}[1]{{||#1||}}

\newcommand{\wtilde}[1]{{\widetilde{#1}}}
\def\supp{\mathop{\mathrm{supp}}\nolimits}

\def\Ran{\mathop{\mathrm{Ran}}\nolimits}
\def\Ker{\mathop{\mathrm{Ker}}\nolimits}

\def\Re{\mathop{\mathrm{Re}}\nolimits}
\def\Im{\mathop{\mathrm{Im}}\nolimits}
\def\loc{\mathop{\mathrm{loc}}\nolimits}

\def\sgn{\mathop{\mathrm{sgn}}\nolimits}
\def\dist{\mathop{\mathrm{dist}}\nolimits}

\def\R{{\mathbb{R}}}
\def\Z{{\mathbb{Z}}}
\def\N{{\mathbb{N}}}
\def\C{{\mathbb{C}}}

\def\S{{\mathcal{S}}}
\def\F{{\mathcal{F}}}
\def\H{{\mathcal{H}}}
\def\A{{\mathcal{A}}}
\def\B{{\mathcal{B}}}

\def\X{{\mathcal{X}}}

\def\<{{\langle}}
\def\>{{\rangle}}

\def\ep{{\varepsilon}}

\title
{Uniform Sobolev estimates for Schr\"odinger operators with scaling-critical potentials and applications}
\author{Haruya Mizutani
}

\AtEndDocument{\bigskip{\footnotesize\textsc{Department of Mathematics, Graduate School of Science, Osaka University, Toyonaka, Osaka 560-0043, Japan.} \par\textit{E-mail address}: \texttt{haruya@math.sci.osaka-u.ac.jp} \par}}

\date{\empty}

\begin{document}
\maketitle

\begin{abstract}
We prove uniform Sobolev estimates for the resolvent of Schr\"odinger operators with large scaling-critical potentials without any repulsive condition. As applications, global-in-time Strichartz estimates including some non-admissible retarded estimates, a H\"ormander type spectral multiplier theorem, and Keller type eigenvalue bounds with complex-valued potentials are also obtained. 
\end{abstract}


%

\section{Introduction and main results}
\label{section_Introduction}

This paper is a continuation of \cite{BoMi,Miz1} where uniform estimates for the resolvent $(H-z)^{-1}$ of the Schr\"odinger operator $H=-\Delta+V(x)$ on $\R^n$ with a real-valued potential $V(x)$ exhibiting one critical singularity were investigated under some {\it repulsive} conditions so that $H$ is non-negative and its spectrum $\sigma(H)$ is purely absolutely continuous. In the present paper we improve upon and extend  those previous results to a class of scaling-critical potentials without any repulsive condition such that $H$ may have (finitely many) negative eigenvalues and multiple scaling-critical singularities. Applications to Strichartz estimates, a H\"ormander type multiplier theorem for $H$ and eigenvalue bounds for $H+W$ with complex potential $W$ are also established.

We first recall some known results in the free case, $H=-\Delta$, describing the motivation of this paper. The classical Hardy-Littlewood-Sobolev (HLS for short) inequality 
 states that
\begin{align*}
\norm{(-\Delta)^{-s/2}f}_{L^q}\le C \norm{f}_{L^p}
\end{align*}
for $f\in \S(\R^n)$, $0<s<n$, $1<p<q<\infty$ and $1/p-1/q=s/n$, where $\S(\R^n)$ denotes the space of Schwarz functions, $(-\Delta)^{-s/2}=\F^{-1}|\xi|^{-s}\F$ is the Riesz potential of order $s$ and $\F$ stands for the Fourier transform in $\R^n$. 
An equivalent form is Sobolev's inequality 
\begin{align*}
\norm{f}_{L^q}\le C\norm{(-\Delta)^{s/2}f}_{L^p}. 
\end{align*}
When $s=2$, the HLS inequality can be regarded as the $L^p$-$L^q$ boundedness of the free resolvent $(-\Delta-z)^{-1}$ at $z=0$. In this context, the HLS inequality was extended to non-zero energies $z\neq0$ by Kenig-Ruiz-Sogge \cite{KRS}, Kato-Yajima \cite{KaYa} and Guti\'errez \cite{Gut} as follows: 
\begin{proposition}[Uniform Sobolev estimates]
\label{proposition_free_Sobolev_1}
Let $n\ge3$, $1\le r\le \infty$ and $(p,q)$ satisfy
\begin{align}
\label{p_q}
\frac{2}{n+1}\le \frac1p-\frac1q\le \frac2n,\quad \frac{2n}{n+3}< p<\frac{2n}{n+1},\quad \frac{2n}{n-1}<q<\frac{2n}{n-3}.
\end{align}
Then the free resolvent $R_0(z)=(-\Delta-z)^{-1}$ satisfies
\begin{align}
\label{proposition_free_Sobolev_1_1}
\norm{R_0(z)f}_{L^{q,r}}&\le C |z|^{\frac n2(\frac1p-\frac1q)-1}\norm{f}_{L^{p,r}}
\end{align}
uniformly in $f\in L^{p,r}(\R^n)$, $z\in \C\setminus[0,\infty)$ and $r$, where $L^{p,r}(\R^n)$ denotes the Lorentz space. 
\end{proposition}


\begin{proof}[Sketch of proof]
By virtue of real interpolation (see Theorem \ref{theorem_interpolation_2} in Appendix \ref{appendix_interpolation}), we may replace without loss of generality $L^{p,r}$ and $L^{q,r}$ by $L^{p}$ and $L^q$, respectively. 
Then the case $1/p+1/q=1$ was proved independently by \cite[Theorem 2.3]{KRS} and \cite[(3.29) in pages 493]{KaYa}; the case $1/p-1/q=2/n$ is due to \cite[Theorems 2.2]{KRS}; otherwise, we refer to \cite[Theorem 6]{Gut}. 
\end{proof}

Note that, when $1/p-1/q=2/n$, the estimate is uniform in $z$  as its name suggests. 

Uniform Sobolev estimates can be used in the study of broad areas including the spectral and scattering theory for Schr\"odinger operators. In \cite{KRS}, the authors applied \eqref{proposition_free_Sobolev_1_1} to study unique continuation properties of $-\Delta+V$ with $V\in L^{n/2}$. In \cite{KaYa,GoSc,IoSc},  \eqref{proposition_free_Sobolev_1_1} was used to show the limiting absorption principle and asymptotic completeness of wave operators for $-\Delta+L$ with a large class of singular perturbations $L$. In \cite{Fra1}, \eqref{proposition_free_Sobolev_1_1} was used to prove the Keller type inequality for $-\Delta+W(x)$ with a complex potential $W\in L^{p}$ with some $p\ge n/2$, which is a quantitative estimate of the spectral radius of $\sigma_{\mathrm{p}}(-\Delta+W)$. In \cite{Gut}, \eqref{proposition_free_Sobolev_1_1} was applied to show the existence of $L^q$-solutions for the stationary Ginzburg-Landau equation under some radiation condition. 

In a more abstract setting, the following observations are satisfied for not only $\Delta$ but also a  general non-negative self-adjoint operator $L$ on $L^2(X,\mu)$: 
\begin{itemize}
\item the uniform Sobolev estimate with $p=\frac{2n}{n+2}$ and $q=\frac{2n}{n-2}$ implies that, for any $w\in L^n$, the weighted resolvent $w(L-z)^{-1}w$ is bounded on $L^2$ uniformly in $z\in \C\setminus[0,\infty)$. As observed by \cite{Kat,KaYa,RoSc}, such a weighted estimate is closely connected with dispersive properties of the solution to \eqref{Cauchy} such as Kato-smoothing effects, time-decay and Strichartz estimates which are fundamental tools in the study of nonlinear Schr\"odinger equations (see \cite{Tao});
\item uniform Sobolev estimates imply that the spectral measure $dE_{L}(\lambda)$ associated with $L$ is bounded from $L^p$ to $L^{p'}$ for $\frac{2n}{n+2}\le p\le \frac{2(n+1)}{n+3}$. This is an important input to prove the H\"ormander type theorem on the $L^p$ boundedness of the spectral multiplier $f(L)$ (see \cite{COSY}). 
\end{itemize}

Motivated by those observations, we are interested in extending \eqref{proposition_free_Sobolev_1_1} to the Schr\"odinger operator $H=-\Delta+V(x)$. If $V$ is of very short range type in the sense that, with some $\ep>0$,
\begin{align}
\label{very_short_range}
|V(x)|\le C(1+|x|)^{-2-\ep},\quad x\in \R^n,
\end{align}
then there is a vast literature on uniform weighted $L^2$-estimates for $(H-z)^{-1}$ without any additional repulsive condition such as suitable smallness of the negative part of $V$ (see, {\it e.g.}, \cite{JeKa,RoTa} and references therein).  Weighted $L^2$-estimates were also obtained for a class of potentials satisfying $|x|^2V\in L^\infty$ under some additional repulsive conditions (\cite{BPST2,BVZ}). In our previous works \cite{BoMi,Miz1}, we proved uniform Sobolev estimates for $H$ with a class of critical potentials $V\in L^{n/2,\infty}$ under some repulsive conditions so that $H$ has purely absolutely continuous spectrum. However, in these literatures, the range of $(p,q)$ has been restricted on the line $1/p+1/q=1$. Furthermore, the situation for (large) critical potentials without any repulsive condition is less understood. 

The main goal of this paper is to prove the full set of uniform Sobolev estimates for $H=-\Delta+V(x)$ with a large scaling-critical potential $V\in L^{n/2,\infty}_0$ without any repulsive condition. The following three types of applications are also established in  the paper: (i) we prove global-in-time Strichartz estimates for the Schr\"odinger equation, 
\begin{align}
\label{Cauchy}
i\partial_t u(t,x)=Hu(t,x)+F(t,x),\ (t,x)\in \R^{1+n};\quad u(0,x)=\psi,\ x\in\R^n,
\end{align}
for all admissible cases and several non-admissible cases; (ii) a H\"ormander type spectral multiplier  theorem for $f(H)$ is obtained provided that $H$ is non-negative; (iii) we obtain Keller type estimates for the eigenvalues (including possible embedded eigenvalues) of the operator $H+W$ with complex potentials $W\in L^p$, $n/2<p\le (n+1)/2$. 

Finally, we mention that the results in this paper could be used to study spectral and scattering theory for both linear and nonlinear Schr\"odinger equations with potentials $V\in L^{n/2,\infty}_0$. 
\\\\
{\it Notation}. 
$A\lesssim B$ (resp. $A\gtrsim B$) means $A\le cB$ (resp. $A\ge cB$) with some universal constant $c>0$. $\<x\>$ stands for $\sqrt{1+|x|^2}$. $\C^\pm:=\{z\in \C\ |\ \pm\Im z>0\}$. Given two Banach spaces $X$ and $Y$, $\mathbb B(X,Y)$ is the Banach space of bounded linear operators from $X$ to $Y$ and $\mathbb B(X)=\mathbb B(X,X)$; $\mathbb B_\infty(X,Y)$ and $\mathbb B_\infty(X)$ are families of compact operators. $\<f,g\>=\int f\overline gdx$ denotes the inner product in $L^2$. We also use the same notation $\<\cdot,\cdot\>$ for the dual coupling between $L^p$ and $L^{p'}$, where $p'=p/(p-1)$ denotes the H\"older conjugate of $p$. $L^p_t\X_x=L^p(\R;\X)$ is the Bochner-Lebesgue space with norm $\norm{F}_{L^p_t\X}=\norm{\norm{F(t,x)}_{\X_x}}_{L^p_t}$. $L^p_TL^q_x:=L^p([-T,T];L^q(\R^n))$. Let $\<\cdot,\cdot\>_T$ be the inner product in $L^2_TL^2_x$ defined by 
$$
\<F,G\>_T=\int_{-T}^T\<F(\cdot,t),G(\cdot,t)\>dt. 
$$
$\H^s(\R^n)$ and $\dot\H^s(\R^n)$ are inhomogeneous and homogeneous $L^2$-Sobolev spaces, respectively. $\mathcal W^{s,p}(\R^n)$ is the $L^p$-Sobolev space. $L^{p,q}(\R^n)$ denotes the Lorentz space (see Appendix \ref{appendix_interpolation}). 

\subsection{Main results}
Throughout the paper we assume that $n\ge3$ and that $V\in L^{n/2,\infty}_0(\R^n)$ is a real-valued function, where $L^{p,\infty}_0(\R^n)$ is the completion of $C_0^\infty(\R^n)$ with respect to the norm $\norm{\cdot}_{L^{p,\infty}}$. It follows from 
H\"older's and Sobolev's inequalities for Lorentz norms (see Appendix \ref{appendix_interpolation}) that $V$ is $\Delta$-form compact. Then the KLMN theorem (\cite[Theorem X.17]{ReSi}) yields that there exists a unique lower semi-bounded self-adjoint operator $H$ on $L^2(\R^n)$ with form domain $\H^1(\R^n)$ such that $$\<Hu,v\>=\<(-\Delta+V)u,v\>,\quad u\in D(H),\ v\in \H^1(\R^n)$$
and that its domain $D(H)=\{u\in \H^1(\R^n)\ |\ Hu\in L^2(\R^n)\}$ is dense in $\H^1(\R^n)$. In other words, $H$ is defined as the Friedrichs extension of the  sesquilinear form $\<(-\Delta+V)u,v\>$. 

\begin{remark}
Note that $L^{n/2,q}\hookrightarrow L^{n/2,\infty}_0$ for all $1\le q<\infty$. Also note that the class $L^{n/2,\infty}_0$ is scaling-critical in the sense that the norm $\norm{V}_{L^{n/2,\infty}}$ is invariant under the scaling $V\mapsto V_\lambda$, where $V_\lambda(x)=\lambda^2V(\lambda x)$. In particular, if $V$ itself is invariant under this scaling, the potential energy $\<Vu,u\>$ has the same scale invariant structure as that for the kinetic energy $\<-\Delta u,u\>$. 
\end{remark}

Let $\mathcal E\subset \sigma(H)$ be the exceptional set of $H$, the set of all eigenvalues and resonances of $H$ (see Definition \ref{definition_resonance}). Note that $\mathcal E\cap (-\infty,0)=\sigma_{\mathrm{d}}(H)$, the discrete spectrum of $H$, and that $\mathcal E$ is bounded in $\R$ (see Remark \ref{remark_weighted_4}). For the absence of embedded eigenvalues and resonances, we have the following simple criterion (see also Remark \ref{remark_multiplier_1}): 

\begin{lemma}
\label{lemma_absence_1}Let $V$ be as above. Then the following statements are satisfied. 
\begin{itemize}
\item[{\rm (1)}] If $V\in L^{n/2}$ then there is no positive eigenvalues and resonances; that is, $\mathcal E\cap (0,\infty)=\emptyset$; 
\item[{\rm (2)}] If $-\Delta+V\ge -\delta\Delta$ with some $\delta>0$ in the sense of forms on $C_0^\infty$ then $0\notin \mathcal E$. 
\end{itemize}
\end{lemma}

\begin{proof}
The proof will be given in Subsection \ref{subsection_resonance}
\end{proof}

Define $
\mathcal E_\delta:=\{z\in \C\ |\ \dist(z,\mathcal E)<\delta\}
$ if $\mathcal E\neq\emptyset$ and $\mathcal E_\delta:=\emptyset$ if $\mathcal E=\emptyset$. For $z\in \C\setminus \sigma(H)$, $R(z)=(H-z)^{-1}$ denotes the resolvent of $H$. 

Then the main result in this paper is as follows. 

\begin{theorem}
\label{theorem_KRS_1}
Suppose that $(p,q)$ satisfies \eqref{p_q}. Then $R(z)$ extends to a bounded operator from $L^{p,2}$ to $L^{q,2}$ for all $z\in \C\setminus\sigma(H)$. Moreover, for any $\delta>0$ there exists $C_\delta>0$ such that 
\begin{align}
\label{theorem_KRS_1_1}
\norm{R(z)f}_{L^{q,2}}&\le C_\delta |z|^{\frac n2(\frac1p-\frac1q)-1}\norm{f}_{L^{p,2}}
\end{align}
 for all $z\in \C\setminus([0,\infty)\cup\mathcal E_{\delta})$ and $f\in L^{p,2}$. In particular, if  $\mathcal E=\emptyset$, then \eqref{theorem_KRS_1_1} holds uniformly with respect to $z\in\C\setminus[0,\infty)$ and $f\in L^{p,2}$.  
\end{theorem}

As a corollary, the limiting absorption principle  in the same topology is derived. 

\begin{corollary}
\label{corollary_KRS_1}
Let $(p,q)$ satisfy \eqref{p_q}. Then the following statements are satisfied.
\begin{itemize}
\item[{\rm (1)}] The boundary values $R(\lambda\pm i0)=\lim\limits_{\ep\searrow0}R(\lambda\pm i\ep)\in \mathbb B(L^{p,2},L^{q,2})$ exist for all $\lambda\in (0,\infty)\setminus\mathcal E$. Moreover, for any $\delta>0$ there exists $C_\delta>0$ such that
\begin{align}
\label{corollary_KRS_1_1}
\norm{R(\lambda\pm i0)f}_{L^{q,2}}\le C_\delta \lambda^{\frac n2(\frac1p-\frac1q)-1}\norm{f}_{L^{p,2}},\quad f\in L^{p,2}(\R^n),\ \lambda\in (0,\infty)\setminus\mathcal E_\delta. 
\end{align}
In particular, if 
$\mathcal E\cap[0,\infty)=\emptyset$, then \eqref{corollary_KRS_1_1} holds uniformly in $\lambda>0$.
\item[{\rm (2)}] 
Assume in addition that $1/p-1/q=2/n$ and $0\notin \mathcal E$. Then $R(0\pm i0)\in \mathbb B(L^{p,2},L^{q,2})$ exist and $R(0+i0)=R(0-i0)$. Moreover, $HR(0+i0)f=f$ and $R(0+i0)Hg=g$ for all $f,g\in \S$ in the sense of distributions. In particular, one has the HLS type inequality
\begin{align}
\label{corollary_KRS_1_2}
\norm{H^{-1}f}_{L^{q,2}}\le C\norm{f}_{L^{p,2}},\quad f\in L^{p,2}(\R^n). 
\end{align}
\end{itemize}
\end{corollary}

As a byproduct of Theorem \ref{theorem_KRS_1}, we also obtain the $L^p$-$L^q$ boundedness of $R(z)$ for fixed $z$ with a wider range than \eqref{p_q}. 

\begin{corollary}
\label{corollary_KRS_2}
For any $z\in \C\setminus\sigma(H)$, $R(z)$ is bounded from $L^{p,2}$ to $L^{q,2}$ whenever 
\begin{align}
\label{corollary_KRS_2_1}
0\le \frac1p-\frac1q\le \frac 2n,\quad \frac{2n}{n+3}<p,q<\frac{2n}{n-3}. 
\end{align}
In particular, $D(H)\subset D(w)$ for any $w\in L^{n/s,\infty}$ with $0\le s<3/2$. Here $D(w)$ denotes the domain of the multiplication operator by $w(x)$. 
\end{corollary}

\begin{remark}
Since $L^{p}\hookrightarrow L^{p,2}$ and $L^{q,2}\hookrightarrow L^{q}$ if $p\le 2\le q$, one has $\mathbb B(L^{p,2},L^{q,2})\subset \mathbb B(L^p,L^q)$. Moreover, 
by virtue of real interpolation (see Theorem \ref{theorem_interpolation_2}), Theorem \ref{theorem_KRS_1}, Corollaries  \ref{corollary_KRS_1} and \ref{corollary_KRS_2} also hold with $L^{p,2}$ and $L^{q,2}$ replaced respectively by $L^{p,r}$ and $L^{q,r}$ for any $1\le r\le \infty$. 
\end{remark}

As explained in the introduction, the resolvent $R(z)$ has a close relation with the spectral measure $E_H$ associated with $H$ through Stone's formula
\begin{align}
\label{Stone}
E_H'(\lambda)=\frac{1}{2\pi i}\lim_{\ep\searrow0}\Big(R(\lambda+i\ep)-R(\lambda-i\ep)\Big),\ \lambda\in (0,\infty)\setminus\sigma_{\mathrm{p}}(H) 
\end{align}
where $E_H'(\lambda)=(dE_H/d\lambda)(\lambda)$ is the density of $E_H$. Using this formula and above theorems, we also obtain the following restriction type estimates. 

\begin{theorem}
\label{theorem_spectral_measure_1}
Assume that $\mathcal E\cap [0,\infty)=\emptyset$. Then, for any $\frac{2n}{n+3}<p\le \frac{2(n+1)}{n+3}$, 
\begin{align}
\label{theorem_spectral_measure_1_1}
\norm{E_H'(\lambda)}_{\mathbb B(L^p,L^{p'})}\le C\lambda^{\frac n2(\frac1p-\frac{1}{p'})-1},\quad\lambda>0. 
\end{align}
\end{theorem}



\begin{remark}
When $V\in L^p$ with $\frac n2\le p\le \frac{n+1}{2}$, the existence of $R(\lambda\pm i0)$ in $\mathbb B(L^{\frac{2(n+1)}{n+3}},L^{\frac{2(n+1)}{n-1}})$ for each $\lambda>0$ was proved by \cite{IoSc}. The uniform estimate \eqref{corollary_KRS_1_1} in the high energy regime $\lambda \ge \lambda_0>0$ was obtained by \cite{GoSc} for the case when $n=3$, $V\in L^{3/2}\cap L^{r}$ with $r>3/2$ and $(p,q)=(4/3,4)$. Recently, \eqref{corollary_KRS_1_1} for $\lambda>0$ and $(p,q)=(\frac{2(n+1)}{n-1},\frac{2(n+1)}{n+3})$ was proved by \cite{HYZ} provided that $V\in L^{n/2}\cap L^{n/2+\ep}$ and $0\notin\mathcal E$ (note that, in this case, $\mathcal E\cap(0,\infty)=\emptyset$ as in Lemma \ref{lemma_absence_1}). Compared with those previous literatures, main new contributions of Theorem \ref{theorem_KRS_1} and Corollary \ref{corollary_KRS_1} are threefold. At first, we obtain the uniform estimates \eqref{theorem_KRS_1_1} and \eqref{corollary_KRS_1_1} with respect to $z$ or $\lambda$ in both high and low energy regimes, under the condition $\mathcal E\cap[0,\infty)=\emptyset$. This is an important input to prove global-in-time Strichartz estimates without any low or high energy cut-off. 
Next, the full set of uniform Sobolev estimates is obtained, while the above previous references considered the case $1/p+1/q=1$ only. In particular, \eqref{theorem_KRS_1_1} and \eqref{corollary_KRS_1_1}  for $(p,q)$ away from the line $1/p+1/q=1$ seems to be new even under the condition \eqref{very_short_range}. Such ``off-diagonal" estimates play an important role in the proof of Strichartz estimates for non-admissible pairs and $L^p$-boundedness of the spectral multiplier $f(H)$ for a wider range of $p$ than that obtained by the ``diagonal" estimate on the line $1/p+1/q=1$ (see Sections \ref{section_Strichartz} and \ref{section_multiplier}, respectively). Finally, we obtain the above results for large critical potentials $V\in L^{n/2,\infty}_0$ without any additional regularity or repulsive condition. Concerning $L^p$-$L^q$ boundedness of $R(z)$ for each $z\in \C\setminus [0,\infty)$, a similar result as Corollary \ref{corollary_KRS_2} was previously obtained by Simon \cite{Sim} for Kato class potentials. However, to our best knowledge, this corollary seems to be new for the present class of potentials. 
\end{remark}

In this paper we also study several applications of the above resolvent estimates to the time-dependent problem, Harmonic analysis and spectral theory associated with $H$. 

We first consider global-in-time estimates for the Schr\"odinger equation \eqref{Cauchy}. Let $e^{-itH}$ be the unitary group generated by $H$ via Stone's theorem. For $F\in L^1_{\mathrm{loc}}(\R;L^2(\R^n))$, we define
$$
\Gamma_HF(t)=\int_0^t e^{-i(t-s)H}F(s)ds. 
$$
For $\psi\in L^2(\R^n)$ and $F\in L^1_{\mathrm{loc}}(\R;L^2(\R^n))$, a unique (mild) solution to \eqref{Cauchy} is then given by
\begin{align}
\label{Duhamel}
u=e^{-itH}\psi-i\Gamma_HF.
\end{align}
The next theorem generalize a result by \cite{BeKl} where the case when $|V(x)|\lesssim \<x\>^{-2-\ep}$  was considered. 

\begin{theorem}	
\label{theorem_smoothing_1}
Assume that $\mathcal E\cap[0,\infty)=\emptyset$. Then, for any $\rho>1/2$, 
$$
\norm{\<x\>^{-\rho}|D|^{1/2}e^{-itH}P_{\mathrm{ac}}(H)\psi}_{L^2_tL^2_x}\le C_\rho\norm{\psi}_{L^2_x},
$$
where $P_{\mathrm{ac}}(H)$ is the projection onto the absolutely continuous subspace associated with $H$. 
\end{theorem}

To state the result on Strichartz estimates, we recall a standard notation. 
\begin{definition}						
\label{definition_admissible}
When $n\ge3$, a pair $(p,q)\in \R^2$ is said to be admissible if 
\begin{align}
\label{admissible}
p,q\ge2,\quad 2/p=n(1/2-1/q).
\end{align}
\end{definition}

\begin{theorem}	
\label{theorem_Strichartz_1}
Suppose that $\mathcal E\cap[0,\infty)=\emptyset$. Then, for any admissible pairs $(p_1,q_1)$ and $(p_2,q_2)$, the solution $u$ to \eqref{Cauchy} satisfies
\begin{align}
\label{theorem_Strichartz_1_1}
\norm{P_{\mathrm{ac}}(H) u}_{L^{p_1}_tL^{q_1}_x}
\lesssim \norm{\psi}_{L^2}+ \norm{F}_{L^{p_2'}_tL^{q_2'}_x},\quad \psi\in L^2,\ F\in L^{p_2'}_tL^{q_2'}_x. 
\end{align} 
For any $\frac{n}{2(n-1)}\le s\le \frac{3n-4}{2(n-1)}$, we also obtain non-admissible inhomogeneous Strichartz estimates:
\begin{align}
\label{theorem_Strichartz_1_2}
\norm{\Gamma_HP_{\mathrm{ac}}(H) F}_{L^2_tL^{\frac{2n}{n-2s}}_x}
\lesssim \norm{F}_{L^2_tL^{\frac{2n}{n+2(2-s)}}_x},\quad F\in L^{2}_tL^{\frac{2n}{n+2(2-s)}}_x. 
\end{align}
\end{theorem}

\begin{remark}
For the admissible case or the case when $\frac{n}{2(n-1)}<s<\frac{3n-4}{2(n-1)}$ we can actually obtain stronger estimates
\begin{align*}
\norm{P_{\mathrm{ac}}(H) u}_{L^{p_1}_tL^{q_1,2}_x}
&\lesssim \norm{\psi}_{L^2}+\norm{F}_{L^{p_2'}_tL^{q_2',2}_x},\\
\norm{\Gamma_HP_{\mathrm{ac}}(H) F}_{L^2_tL^{\frac{2n}{n-2s},2}_x}
&\lesssim \norm{F}_{L^2_tL^{\frac{2n}{n+2(2-s)},2}_x},\quad\frac{n}{2(n-1)}< s< \frac{3n-4}{2(n-1)},
\end{align*}
than \eqref{theorem_Strichartz_1_1} and \eqref{theorem_Strichartz_1_2}. 
Inhomogeneous estimates for some other non-admissible pairs may be also deduced from \eqref{theorem_Strichartz_1_2} and usual inhomogeneous estimates. For instance, if we interpolate between \eqref{theorem_Strichartz_1_2} and the trivial estimate $\norm{\Gamma_HP_{\mathrm{ac}}(H)F}_{L^\infty_tL^2_x}\le \norm{F}_{L^1_tL^2_x}$ then
$$
\norm{\Gamma_HP_{\mathrm{ac}} F}_{L^p_tL^q_x}
\lesssim \norm{F}_{L^{\tilde p'}_tL^{\tilde q'}_x},
$$
where $\frac{n}{2(n-1)}\le s\le \frac{3n-4}{2(n-1)}$ and 
$
\frac ns(\frac12-\frac1q)=\frac2p=\frac{2}{\tilde p}=\frac{n}{2-s}(\frac12-\frac{1}{\tilde q})$. Inhomogeneous Strichartz estimates with non-admissible pairs for the free Schr\"odinger equation have been studied by several authors \cite{Kat3,KeTa,Fos,Vil,KoSe} under suitable conditions on $(p,q)$ (see \cite{Fos,KoSe}). The estimates \eqref{theorem_Strichartz_1_2} correspond to the endpoint cases for such conditions.  It is also worth noting that, as well as the estimates for admissible pairs, non-admissible estimates can be used in the study of nonlinear Schr\"odinger equations (see \cite{Kat3}). 
\end{remark}

\begin{remark}
There is a vast literature on Strichartz estimates for Schr\"odinger equations with potentials. We refer to \cite{RoSc,Gol,Bec,BoMi} and reference therein. We also note that the dispersive ($L^1$-$L^\infty$) estimate for $e^{-itH}P_{\mathrm{ac}}(H)$ and $L^p$-boundedness of wave operators $W_{\pm}$, which imply Strichartz estimates, have been also extensively studied (see \cite{RoSc,BeGo,Yaj1,Bec2} and reference therein). In particular, Goldberg \cite{Gol} proved the endpoint Strichartz estimates for $e^{-itH}P_{\mathrm{ac}}$ under the conditions that $V\in L^{n/2}$, $0\notin \mathcal E$ and $n\ge3$. When $n=3$, Strichartz estimates for all admissible cases and some non-admissible cases (which are different from \eqref{theorem_Strichartz_1_2}) for $V\in L^{3/2,\infty}_0$ were obtained by Beceanu \cite{Bec}. Compared with those previous literatures, a new contribution of this theorem is that we obtain the full set of admissible Strichartz estimates \eqref{theorem_Strichartz_1_1} including the inhomogeneous double endpoint case for all $n\ge3$. Moreover, non-admissible estimates \eqref{theorem_Strichartz_1_2} are new even for $V\in L^{n/2}$. 
\end{remark}

The next application of resolvent estimates in this paper is the $L^p$-boundedness of the spectral multiplier $F(H)$, which is defined by the spectral decomposition theorem, namely
$$
F(H)=\int_{\sigma(H)}F(\lambda)dE_H(\lambda),
$$
For the free case $H=-\Delta$, H\"ormander's multiplier theorem \cite{Hor} implies that if $F\in L^\infty$ satisfies
\begin{align}
\label{assumption_multiplier}
\sup_{t>0}\norm{\psi(\cdot)F(t\cdot)}_{\H^\beta}<\infty
\end{align}
with some nontrivial $\psi\in C_0^\infty(\R)$ supported in $(0,\infty)$ and $\beta>n/2$, then $F(-\Delta)$ is bounded on $L^p$ for all $1<p<\infty$. The following theorem is a generalization of this result to non-negative Schr\"odinger operators with scaling-critical potentials. 

\begin{theorem}	
\label{theorem_multiplier_1}
Suppose that $\mathcal E\cap[0,\infty)=\emptyset$ and $H\ge0$. Then, for any $F\in L^\infty(\R)$ satisfying \eqref{assumption_multiplier} with some nontrivial $\psi\in C_0^\infty(\R)$ supported in $(0,\infty)$ and $\beta>3/2$, $F(\sqrt H)$ is bounded on $L^p$ for all $2n/(n+3)<p<2n/(n-3)$ and satisfies
\begin{align}
\label{theorem_multiplier_1_1}
\norm{F(\sqrt H)}_{\mathbb B(L^p)}\le C(\sup_{t>0}\norm{\psi(\cdot)F(t\cdot)}_{\H^\beta}+|F(0)|). 
\end{align}
\end{theorem}

It is easy to check that $F$ satisfies \eqref{assumption_multiplier} if and only if $G(\lambda)=F(\lambda^2)$ does. Therefore, \eqref{theorem_multiplier_1_1} also holds with $F(\sqrt{H})$ replaced by $F(H)$. Also note that, in the proof of this theorem, the restriction estimates \eqref{theorem_spectral_measure_1_1} will play an essential role and the restriction for the range of $p$ when $n\ge4$ is due to the condition $p>\frac{2n}{n+3}$ for \eqref{theorem_spectral_measure_1_1}. 

\begin{remark}
Some applications of Theorem \ref{theorem_multiplier_1} will be also established (see Section \ref{section_multiplier}). At first we obtain the equivalence between Sobolev norms $\norm{(-\Delta)^{s/2}u}_{L^2}$ and $\norm{H^{s/2}u}_{L^2}$  for $0\le s<3/2$. Secondly, we shall prove square function estimates for the Littlewood-Paley decomposition via the spectral multiplier associated with $H$ . These are known to play an important role in the study of nonlinear Schr\"odinger equations with potentials (see, {\it e.g.}, \cite{KMVZZ}). 
\end{remark}

\begin{remark}
If the Schr\"odinger semigroup $e^{-tH}$ satisfies the Gaussian estimate or some generalized Gaussian type estimates, then H\"ormander's multiplier theorem for $F(H)$ have been extensively studied (see \cite{COSY} and reference therein). Compared with such cases, the interest of Theorem \ref{theorem_multiplier_1} is that we obtain H\"ormander's multiplier theorem under a scaling-critical condition $V\in L^{n/2,\infty}_0$, while it is not known for such a class of potentials whether $H$ satisfies (generalized) Gaussian estimates or not, even if $H$ is assumed to be non-negative. 
\end{remark}

\begin{remark}
\label{remark_multiplier_1}
To ensure the non-negativity of $H$, it suffices to assume $\norm{V_-}_{L^{n/2,\infty}}\le S_n^{-1}$, where $V_-=\max\{0,-V\}$ is the negative part of $V$ and $$S_n:=\frac{n(n-2)}{4}2^{\frac2n}\pi^{1+1/n}\Gamma\Big(\frac{n+1}{2}\Big)^{-\frac2n}$$ is the best constant in Sobolev's inequality. $\norm{f}_{L^{\frac{2n}{n-2}}}\le S_n\norm{\nabla f}_{L^2}$. 
Moreover, if $\norm{V_-}_{L^{n/2}}< S_n^{-1}$ then $0\notin \mathcal E$ by Lemma \ref{lemma_absence_1}. 
\end{remark}

The last application of Theorem \ref{theorem_KRS_1} in the paper is the Keller type inequality for individual eigenvalues of a non-self-adjoint Schr\"odinger operator. Let $0<\gamma<\infty$ and $W\in L^{n/2+\gamma}(\R^n;\C)$ a possibly complex-valued potential. Then $W$ is $H$-form compact and we define the operator $H_W=H+W$ as a form sum. Under this setting, it is known that $\sigma(H_W)$ is contained in a sector $\{z\in \C\ |\ |\arg(z-z_0)|\le \theta\}$ with some $z_0\in \R$ and $\theta\in[0,\pi/2)$ (see \cite{Kat}), but the point spectrum $\sigma_{\mathrm{p}}(H_W)$ could be unbounded in $\C$ in general even if $V\equiv0$ and $W$ is smooth. 
The following theorem, however, shows that this is not the case if $0<\gamma\le1/2$. 

\begin{theorem}
\label{theorem_EB_1}
Let $\delta>0$. If $0<\gamma\le1/2$, any eigenvalue $E\in \C\setminus\mathcal E_\delta$ of $H_W$ satisfies
\begin{align}
\label{theorem_EB_1_1}
|E|^{\gamma}\le C_{\gamma,\delta}\norm{W}_{L^{\frac n2+\gamma}}^{\frac n2+\gamma}. 
\end{align}
Moreover, if $\gamma>1/2$, any eigenvalue $E\in \C\setminus\mathcal E_\delta$ of $H_W$ satisfies
\begin{align}
\label{theorem_EB_1_2}
|E|^{1/2}\dist(E,[0,\infty))^{\gamma-1/2}\le C_{\gamma,\delta}\norm{W}_{L^{\frac n2+\gamma}}^{\frac n2+\gamma}. 
\end{align}
Here the constant $C_{\gamma,\delta}=C({\gamma,\delta,n,V})>0$ may be taken uniformly in $W$. 
\end{theorem}

\begin{remark}
Theorem \ref{theorem_EB_1} implies the following spectral consequence. If $0<\gamma\le1/2$ then$$
\sigma_{\mathrm{p}}(H_W)\subset \mathcal E_\delta\cup\Big\{z\in \C\ \Big|\ |z|^\gamma\le C_{\gamma,\delta}\norm{W}_{L^{\frac n2+\gamma}}^{\frac n2+\gamma}\Big\}
$$
In particular, since $\mathcal E$ is bounded in $\R$ (see Remark \ref{remark_weighted_4}), $\sigma_{\mathrm{p}}(H_W)$ is bounded in $\C$. On the other hand, if $\gamma>1/2$ and $\Re E>0$, then $E$ satisfies
$$|\Im E|\le C_{\gamma,\delta} |E|^{-\frac{1}{2(\gamma-1/2)}}\norm{W}_{L^{\frac n2+\gamma}}^{\frac{n+2\gamma}{{2\gamma-1}}}.$$ This implies that, for any sequence $\{E_j\}\subset\sigma_{\mathrm{p}}(H_W)\setminus[0,\infty)$ satisfying $\Re E_j\to +\infty$ as $j\to\infty$, we have $|\Im E_j|\to 0$ as $j\to\infty$. 
\end{remark}

\begin{remark}
For a complex potential $W(x)$, the estimates \eqref{theorem_EB_1_1} and \eqref{theorem_EB_1_2} were firstly proved by Frank \cite{Fra1,Fra2} for the case when $-\Delta+W(x)$ and then extended to the operator $-\Delta-a|x|^{-2}+W(x)$ with $a\le (n-2)-2/4$ by \cite{Miz1}. In both cases, the free Hamiltonians $-\Delta$ and $-\Delta-a|x|^{-2}$ are non-negative and purely absolutely continuous. Theorem \ref{theorem_EB_1} shows that the same result still holds even if the free Hamiltonian has (embedded) eigenvalues or resonances. 
\end{remark}

The rest of the paper is devoted to the proof of above results. We here outline the plan of the paper, describing rough idea of proofs. Following the classical scheme, the proof of uniform Sobolev estimates is based on the resolvent identity 
$
R(z)=(I+R_0(z)V)^{-1}R_0(z). 
$

In Section \ref{section_preliminaries} we collect several properties on the free resolvent $R_0(z)$ used throughout the paper and, then, study basic properties of the exceptional set $\mathcal E$. In particular, we show that $R_0(z)V$ extends to a $\mathbb B_\infty(L^q)$-valued continuous function on $\overline{\C^+}$. This fact plays an important role to justify the above resolvent identity. The proof of Lemma \ref{lemma_absence_1} is also given in Section \ref{section_preliminaries}. 

Using materials prepared in Section \ref{section_preliminaries} and the Fredholm alternative theorem, we prove Theorem \ref{theorem_KRS_1}, Corollaries \ref{corollary_KRS_1} and \ref{corollary_KRS_2} and Theorem \ref{theorem_spectral_measure_1} in Section \ref{section_uniform_Sobolev}. 

Section \ref{section_Strichartz} is devoted to proving Theorems \ref{theorem_smoothing_1} and  \ref{theorem_Strichartz_1}. The proof follows an abstract scheme by \cite{RoSc} (see also \cite{BPST2,BoMi}) which is based on Duhamel's formulas
$$
e^{-itH}=e^{it\Delta}-i\Gamma_0V\Gamma_H,\quad \Gamma_H=\Gamma_0-i\Gamma_0V\Gamma_H,
$$
where $\Gamma_0=\Gamma_{-\Delta}$. Using these identities, the proof can be reduced to that of corresponding estimates for the free propagators $e^{it\Delta}$ and $\Gamma_0$ which are well known, and $L^2_tL_x^2$ estimates for $V_1e^{-itH}P_{\mathrm{ac}}(H)$ and $V_1\Gamma_HP_{\mathrm{ac}}(H)V_2$ with a suitable decomposition $V=V_1V_2$. Kato's smooth perturbation theory \cite{Kat} allows us to deduce such $L^2_tL^2_x$-estimates from the resolvent estimate 
$$
\sup_{z\in \C\setminus\R}\norm{V_1R(z)P_{\mathrm{ac}}(H)V_2}_{\mathbb B(L^2)}<\infty,
$$
which follows from uniform Sobolev estimates for $P_{\mathrm{ac}}(H)R(z)$ (which are also proved as a corollary of Theorem \ref{theorem_KRS_1} in the end of Section \ref{section_uniform_Sobolev}) and H\"older's inequality. A rigorous justification of the above Duhamel's formulas in the sense of forms are also given in Section  \ref{section_Strichartz}. 

Proofs of the spectral multiplier theorem and its applications are given in Section \ref{section_multiplier}. The proof of Theorem \ref{theorem_multiplier_1} employs an abstract method by \cite{COSY} which allows us to deduce Theorem \ref{theorem_multiplier_1} from the restriction estimates \eqref{theorem_spectral_measure_1_1}  and the so-called Davies-Gaffney estimate for the Schr\"odinger semigroup $e^{-tH}$. In the proof of the Davies-Gaffney estimate, we use the condition that $H$ is non-negative. 

Section \ref{section_EB} is devoted to the proof of Theorem \ref{theorem_EB_1}, which follows  basically the same line as in \cite{Fra1,Fra2} and is based on the estimates \eqref{theorem_KRS_1_1},  \eqref{corollary_KRS_1_1} and the Birman-Schwinger principle. 

Appendix \ref{appendix_interpolation} is devoted to a brief introduction of real interpolation and Lorentz spaces. 
\\\\
\noindent{\bf Acknowledgements.} The author would like to express his sincere gratitude to Kenji Nakanishi and Jean-Marc Bouclet for valuable discussions. He is partially supported by JSPS KAKENHI Grant Numbers JP25800083 and JP17K14218. 

\section{Preliminaries}
\label{section_preliminaries}
In this section we first study  several properties of the free resolvent, which will often appear in the sequel. The second part is devoted to a detail study of the exceptional set of $H$. 

\subsection{The free resolvent}
\label{subsection_LAP}
For $z\notin \C\setminus[0,\infty)$, $R_0(z)=(-\Delta-z)^{-1}$ denotes the free resolvent, which is defined as a Fourier multiplier with symbol $(|\xi|^2-z)^{-1}$. The integral kernel of $R_0(z)$ is given by\begin{align*}R_0(z,x,y)=\frac i4\Big(\frac{z^{1/2}}{2\pi |x-y|}\Big)^{n/2-1}H^{(1)}_{n/2-1}(z^{1/2}|x-y|),\quad \Im z^{1/2}>0,\end{align*} where $H^{(1)}_{n/2-1}$ is the Hankel function of the first kind. The pointwise estimate $$|H^{(1)}_{n/2-1}(w)|\le C_n\begin{cases}|w|^{-n/2+1}&\text{for}\ |w|\le1,\\|w|^{-1/2}&\text{for}\ |w|>1,\end{cases}$$then  implies that there exists $C_n>0$ depends only on $n$ such that\begin{align}\label{pointwise}|R_0(z,x,y)|\le C_n (|x-y|^{-n+2}+|x-y|^{-\frac{n-1}{2}})\<z\>^{\frac{n-3}{4}}\end{align} (see \cite{Jen_1}). For $s\in \R$, we let $L^2_s=L^2(\R^n,\<x\>^{2s}dx)$ and $\H^{2}_s=\{u\ |\ \partial^\alpha u\in L^2_s,\ |\alpha|\le2\}$. Then the following limiting absorption principle in weighted $L^2$-spaces is well known (see \cite{Agm,JeKa,Jen_1,Jen_2}): 

\begin{lemma}
\label{lemma_free_Sobolev_0}
Let $s>(n+1)/2$. Then $R_0(z)$ is bounded from $L^2_s$ to $L^2_{-s}$ uniformly in $z\in \C\setminus[0,\infty)$.  Moreover, the following statements are satisfied. 
\begin{itemize}
\item Boundary values 
$
R_0(\lambda\pm i0)=\lim\limits_{\ep\to0}R_0(\lambda\pm i\ep)\in \mathbb B_\infty(L^2_s,L^2_{-s})
$
exist on $[0,\infty)$ such that $R_0(0\pm i0)=(-\Delta)^{-1}$. Moreover, $R_0(\lambda\pm i0)\in  \mathbb B_\infty(L^2_s,\H^2_{-s})$ if $\lambda>0$. 
\item Define the extended free resolvent $R_0^\pm(z)$ by $R_0^\pm(z)=R_0(z)$ if $z\in \C\setminus[0,\infty)$ and $R_0^\pm(z)=R_0(z\pm i0)$ if $z\ge0$. Then $R_0^\pm(z)$ are $\mathbb B_\infty(L^2_s,L^2_{-s})$-valued continuous functions on $\overline{\C^\pm}$. 
\item For any $z\in \overline{\C^+}$ and $f\in L^2_s$, $(-\Delta-z)R^\pm_0(z)f=f$ in the sense of distributions. 
\end{itemize}
\end{lemma}


\begin{figure}[htbp]
\label{figure_1}
\begin{center}
\scalebox{0.9}[0.9]{
\begin{tikzpicture}

\draw (0,0) rectangle (6.75,6.75);
\draw[->]  (0,0) -- (0,7.5);
\draw[->]  (0,0) -- (7.5,0);
\draw (6.75,0) node[below] {$\ 1$};
\draw (0,6.75) node[left] {$1$};
\draw (0,0) node[below, left] {$0$};
\draw (7.5,0) node[above] {$\frac1p$};
\draw (0,7.5) node[right] {$\frac1q$};
\draw (7.5,0) node[above] {$\frac1p$};
\draw (0,7.5) node[right] {$\frac1q$};
\draw[dashed] (0,0) -- (6.75,6.75);
\draw[dashed] (6.75,0) -- (0,6.75);

\draw[dashed] (6.3,3) -- (7.8,4.5);
\draw (7,3.5) node[right] {$\frac1p-\frac1q=\frac 2n$};
\draw (7,3.5) node[right] {$\frac1p-\frac1q=\frac 2n$};
\draw[dashed] (5.25,3) -- (7.8,5.55);
\draw (7.9,5.30) node[above] {$\frac1p-\frac1q=\frac{2}{n+1}$};
\draw (7.9,5.30) node[above] {$\frac1p-\frac1q=\frac{2}{n+1}$};


\draw[dashed] (3.75,1.5) -- (3.75,3.0) -- (5.25,3.0);
\draw (3.75,0.45) circle (2pt) node[right] {$A$};    
\draw[dashed] (3.75,0.45) -- (0,0.45);
\draw (6.3,3)  circle (2pt) node[right] {$\!A'$}; 
\draw[dashed] (6.3,6.3) -- (6.3,0);
\draw[line width=1pt] (3.75,0.45) -- (6.3,3); 
\draw[dashed] (3.75,0.45) -- (3.75,1.5) ; 
\draw[dashed] (5.25,3) -- (6.3,3) ; 

\draw (0.45,0.45) circle (2pt) node[above] {$C$};    
\draw (6.3,6.3)  circle (2pt) node[left] {$C'$}; 
\draw[line width=1pt] (0.45,0.45) -- (6.3,6.3); 

\draw (3.75,1.5)  circle (2pt) node[left] {$B$};   
\draw (5.25,3)  circle (2pt) node[above] {$B'$}; 
\draw[line width=1pt] (3.75,1.5) -- (5.25,3); 


\draw[dashed] (3.75,0.45) -- (3.75,0);
\draw (3.75,0) node[below] {$\frac{n+1}{2n}$};
\draw (3.75,0) node[below] {$\frac{n+1}{2n}$};
\draw (0,0.45) node[left] {$\frac{n-3}{2n}$};
\draw (0,0.45) node[left] {$\frac{n-3}{2n}$};
\draw(6.3,0) node[below]  {$\frac{n+3}{2n}$};
\draw(6.3,0) node[below]  {$\frac{n+3}{2n}$};
\draw (0,3) node[left] {$\frac{n-1}{2n}$};
\draw (0,3) node[left] {$\frac{n-1}{2n}$};
\draw[dashed] (5.025,1.725) -- (5.025,0) node[below] {$\frac{n+2}{2n}$};
\draw[dashed] (5.025,1.725) -- (0,1.725) node[left] {$\frac{n-2}{2n}$};
\draw (5.025,0) node[below] {$\frac{n+2}{2n}$};
\draw (0,1.725) node[left] {$\frac{n-2}{2n}$};
\draw[dashed] (0,3) -- (3.75,3.0);
\end{tikzpicture}
}
\end{center}
\caption{The set of $(1/p,1/q)$ satisfying \eqref{p_q} is the trapezium $ABB'A'$ with two closed line segments $\overline{AB}$, $\overline{B'A'}$ removed. The set of $(1/p,1/q)$ satisfying \eqref{corollary_KRS_2_1} is the trapezium $ACC'A'$ with two closed line segments $\overline{AC}$, $\overline{C'A'}$ removed.}
\end{figure}
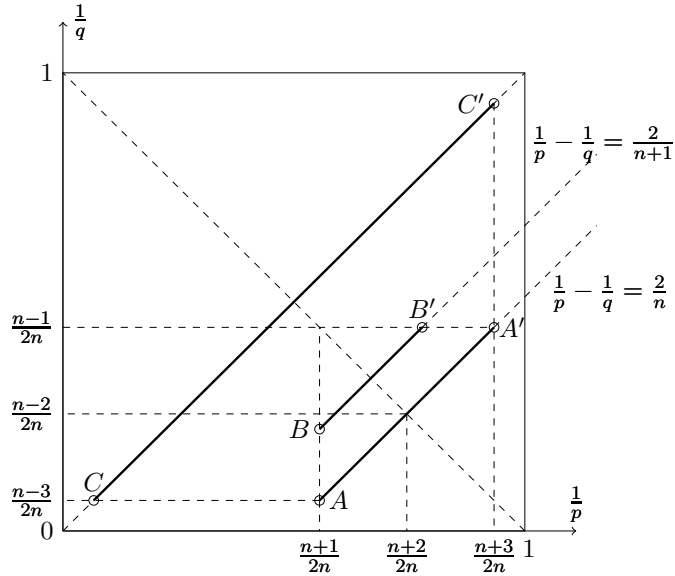

The following corollaries are immediate consequences of Lemma \ref{lemma_free_Sobolev_0} and Proposition \ref{proposition_free_Sobolev_1}. 

\begin{corollary}
\label{corollary_free_Sobolev_1} Let $(p,q)$ satisfy \eqref{p_q} and $2n/(n+3)<r<2n/(n+1)$. Then,
\begin{itemize}
\item[{\rm(1)}] $R_0^\pm(z)$ extend to elements in $\mathbb B(L^{p,2},L^{q,2})$ and satisfy
\begin{align}
\label{corollary_free_Sobolev_1_2}
\norm{R_0^\pm(z)}_{\mathbb B(L^{p,2},L^{q,2})}\le C|z|^{\frac n2(\frac1p-\frac1q)-1},\quad z\in \overline{\C^\pm}\setminus\{0\}.
\end{align}
\item[{\rm(2)}] For any $f\in L^{p,2}$ and $g\in L^{q',2}$, $\<R_0^\pm(z)f,g\>$ are continuous on $\overline{\C^\pm}\setminus\{0\}$.  
\item[{\rm(3)}] For any $z\in \overline{\C^\pm}$ and $f\in L^{r,2}$, $(-\Delta-z)R_0^\pm(z)f=f$ in the sense of distributions. 
\end{itemize}Assuming in addition that $1/p-1/q=2/n$, the statements {\rm (1)} and {\rm (2)} hold for all $z\in \overline{\C^\pm}$. 
\end{corollary}


Throughout the paper, we frequently use the notation
\begin{align}
\label{p_q_s}
p_s=\frac{2n}{n+2(2-s)},\quad q_s=\frac{2n}{n-2s}. 
\end{align}
Note that $\{(p_s,q_s)\ |\ 1/2<s<3/2\}=\{(p,q)\ |\ \text{$(p,q)$ satisfies \eqref{p_q} and $1/p-1/q=2/n$}\}$. 
\begin{corollary}
\label{corollary_free_Sobolev_2}
Let $1/2<s<3/2$, $V_1\in L^{n/s,\infty}_0$ and $V_2\in L^{n/(2-s),\infty}_0(\R^n)$. Then $V_1R_0^\pm(z)V_2$ are $\mathbb B_\infty(L^2)$-valued continuous function of $z\in \overline{\C^\pm}$. 
\end{corollary}

\begin{proof}
Corollary \ref{corollary_free_Sobolev_1} (1) with $(p,q)=(p_s,q_s)$ and H\"older's inequality \eqref{Holder} imply
$$
\sup_{z\in \overline{\C^+}}\norm{V_1R_0^\pm(z)V_2}_{\mathbb B(L^2)}\lesssim \norm{V_1}_{L^{\frac ns,\infty}}\norm{V_2}_{L^{\frac{n}{2-s},\infty}}. 
$$
Since $C_0^\infty$ is dense in $L^{p,\infty}_0$ for all $1<p<\infty$ and an operator norm limit of compact operators is compact, we observe from this uniform bound and a standard $\ep/3$ argument that it suffices to show the corollary for $V_1,V_2\in C_0^\infty$. In this case, the corollary follows from Lemma \ref{lemma_free_Sobolev_0}. 
\end{proof}

The following proposition plays an essential role throughout the paper. 

\begin{proposition}
\label{compactness}
Let $w\in L^{n/2,\infty}_0(\R^n)$, $1/2<s<3/2$ and $q_s$ as above. Then $R_0(z)w\in \mathbb B_\infty(\H^1)$ for all $z\in\C\setminus [0,\infty)$. Moreover, $R_0^\pm(z)w$ are $\mathbb B_\infty(L^{q_s,2})$-valued continuous functions on $\overline{\C^\pm}$. 
\end{proposition}

\begin{remark}
$R_0^\pm(z)w$ are also $\mathbb B_\infty(L^{q_s})$-valued continuous functions on $\overline{\C^\pm}$. The proof is completely same. 
\end{remark}

\begin{proof}
The facts $R_0(z)w\in \mathbb B(\H^1)\cap \mathbb B(L^{q_s,2})$ and $R_0^\pm(z)w\in\mathbb B(L^{q_s,2})$ follow from the continuity $R_0(z):\H^{-1}\to \H^1$, uniform Sobolev estimates \eqref{proposition_free_Sobolev_1_1} and H\"older's inequality for Lorentz norms. 

To prove the compactness and the continuity (in $z$), by virtue of these estimates and the same argument as above, we may assume  without loss of generality that $w\in C_0^\infty$ and $w(x)=0$ for $|x|\ge c_0$ with some $c_0>0$. Then it was proved by \cite[Lemma 4.2]{IoSc} that there is a Banach space $X$ satisfying the continuous embedding $X\hookrightarrow \H^{-1}$ such that $w:X^*\to X$ is compact as a multiplication operator. $R_0(z)w$ is therefore compact on $\H^1$ for $z\in \C\setminus[0,\infty)$.

Next we shall prove that $R_0^\pm(z)w$ are compact on $L^{q_s,2}$ for $z\in \overline{\C^\pm}$. As before, we only consider $R_0^+(z)$. By virtue of real interpolation (Theorem \ref{theorem_interpolation_2}), it suffices to show that $R_0^+(z)w$ is compact on $L^{q_s}$ for all $1/2<s<3/2$. Assume that $f_j\in L^{q_s}$ and $\norm{f}_{L^{q_s}}\le1$. Extracting a subsequence if necessary we may assume $f_j\to 0$ weakly in $L^{q_s}$. Then it remains to show that there exists a subsequence $\{\wtilde f_{j}\}\subset \{f_j\}$ such that $R_0^+(z)w\wtilde f_{j}\to0$ strongly in $L^{q_s}$. 
To this end, we decompose $R_0^+(z)w$ into two regions $B_r^c$ and $B_r$, where $B_r=\{x\in \R^n\ |\ |x|\le r\}$. For the former case, the pointwise estimate \eqref{pointwise} yields
\begin{align*}
|R_0^+(z)wf_j(x)|
\le C_{n}\<z\>^{\frac{n-3}{4}}|x|^{-\frac{n-1}{2}}\norm{wf_j}_{L^1}
\le C_{n,z}|x|^{-\frac{n-1}{2}}\norm{w}_{L^{\frac{2n}{n+2s}}}
\end{align*}
uniformly in $|x|\ge r$, $r\ge 2c_0$ and $j\ge0$. Let us fix $\ep>0$ arbitrarily. Since $$
\norm{|x|^{-\frac{n-1}{2}}}_{L^{q_s}(B_r^c)}\le C r^{-(s-1/2)},
$$
we can find $r_0=r_0({n,\ep,z,w})>0$ such that
\begin{align}
\label{proof_compactness_1}
\norm{R_0^+(z)wf_j}_{L^{q_s}(B_{r_0}^c)}<\ep.
\end{align}
For the latter case, we observe that $R_0^+(z)w:L^{q_s}(\R^n)\to \mathcal W^{2,q_s}(\R^n)$ is bounded since
\begin{align}
\label{proof_compactness_2}
(-\Delta+1)R_0^+(z)wf=(-\Delta-z)R_0^+(z)wf+(z+1)R_0^+(z)wf=wf+(z+1)R_0^+(z)wf
\end{align}
for all $f\in L^{q_s}$ by Corollary \ref{corollary_free_Sobolev_1} (3). In particular, $\{R_0^+(z)wf_j\}_j$ is bounded in $\mathcal W^{2,q_s}(B_{r_0})$. Since $\mathcal W^{2,q_s}(B_{r_0})$ embeds compactly into $L^{q_s}(B_{r_0})$ 
by the Rellich-Kondrachov compactness theorem, one can find a subsequence $\{\wtilde f_j\}\subset \{f_j\}$ such that 
\begin{align}
\label{proof_compactness_3}
\lim_{j\to\infty}\norm{R_0^+(z)w\wtilde f_j}_{L^{q_s}(B_{r_0})}=0.
\end{align}
It follows from \eqref{proof_compactness_1} and \eqref{proof_compactness_3} that
$$
\limsup\limits_{j\to \infty}\norm{R_0^+(z)w\wtilde f_j}_{L^{q_s}(\R^n)}\le \ep. 
$$
By extracting further a subsequence, we conclude that $R_0^+(z)w\wtilde f_j\to0$ strongly in $L^{q_s}$. 

To prove the continuity, let us fix a bounded set $\Lambda\subset \overline{\C^+}$ arbitrarily. We first show that, for any $z,z_j\in \Lambda$ and $g,g_j\in L^{q_s,2}$ satisfying $z_j\to z$ and $g_j\to g$ weakly in $L^{q_s,2}$ as $j\to \infty$, 
\begin{align}
\label{proof_compactness_4}
R_0^+(z_j)wg_j\to R_0^+(z)wg\quad \text{strongly in}\ L^{q_s,2}\ \text{as $j\to\infty$}.
\end{align}
To this end, we write 
$$
R_0^+(z_j)wg_j-R_0^+(z)wg=\Big(R_0^+(z_j)w-R_0^+(z)w\Big)g_j+R_0^+(z)w(g_j-g).
$$
The second term $R_0^+(z)w(g_j-g)$ converges to $0$ strongly in $L^{q_s,2}$ since $R_0^+(z)w$ is compact on $L^{q_s,2}$ and $g_j\to g$ weakly. For the first part, we set $h_j=(R_0^+(z_j)w-R_0^+(z)w\Big)g_j$ and shall show that $h_j\to 0$ strongly in $L^{q_s,2}$. Since $\{g_j\}\subset L^{q_s,2}$ is bounded, say $\norm{g_j}_{L^{q_s,2}}\le M$ with $M>0$ being independent of $j$, we learn by the same argument as above that, with some $\gamma_j=\gamma_j(s,n)>0$, 
$$
\norm{R_0^+(\zeta)wg_j}_{L^{q_s,2}(B_r^c)}\le C_{n,M,w}\<\zeta\>^{\gamma_1}r^{-\gamma_2}
$$
for all $\zeta\in \overline {\C^+}$, $j\ge1$ and $r\ge2c_0$ ,where $C_{n,M,w}$ may be taken uniformly in $j$ and $r$. This estimate yields that, for any $\ep>0$, there exists $0<r_\ep=r(n,M,w,\Lambda,\ep)\sim \ep^{-1/\gamma_2}$ such that
\begin{align}
\label{proof_compactness_5}
\sup_{j\ge1,}\norm{h_j}_{L^{q_s,2}(B_{r_\ep}^c)}\le \sup_{j\ge1}\Big(\norm{R_0^+(z_j)wg_j}_{L^{q_s,2}(B_{r_\ep}^c)}+\norm{R_0^+(z)wg_j}_{L^{q_s,2}(B_{r_\ep}^c)}\Big)<\ep.
\end{align}
On the other hand, it follows from Sobolev's embedding on $\R^n$ that
\begin{align*}
\norm{h_j}_{L^{q_s,2}(B_{r_\ep})}
\le C_{\ep,N}\norm{(-\Delta+1)\<x\>^{-N}h_j}_{L^{2}(\R^n)}
\le C_{\ep,N}\norm{\<x\>^{-N}(-\Delta+1)h_j}_{L^{2}(\R^n)}
\end{align*}
for all $N\ge0$, where we have used the fact that $(-\Delta+1)\<x\>^{-N}(-\Delta+1)^{-1}\<x\>^{N}$ is a pseudodifferential operator of order $0$ and thus bounded on $L^p$ for all $1<p<\infty$.  \eqref{proof_compactness_2} then yields
\begin{align*}
\norm{\<x\>^{-N}(-\Delta+1)h_j}_{L^{2}}
&\le |z-z_j|\norm{\<x\>^{-N}R_0^+(z_j)\<x\>^{-N}}_{\mathbb B(L^2)}\norm{\<x\>^Nwg_j}_{L^2}\\
&+(|z|+1)\norm{\<x\>^{-N}(R_0^+(z_j)-R_0^+(z))\<x\>^{-N}}_{\mathbb B(L^2)}\norm{\<x\>^Nwg_j}_{L^2}.
\end{align*}
Let $N\ge (n+1)/2$. Since $\<x\>^{-N}R_0^+(z)\<x\>^{-N}$ is bounded on $L^2$ uniformly in $z\in \overline{\C^+}$ and continuous on $\overline{\C^+}$ in the operator norm topology of $\mathbb B(L^2)$ by Lemma \ref{lemma_free_Sobolev_0} and $$\norm{\<x\>^Nwg_j}_{L^2}\le CM\norm{\<x\>^Nw}_{L^{\frac{2n}{n+2s},2}}\le C_{N,M,\omega}$$ uniformly in $j$, we see that $\lim\limits_{j\to \infty}\norm{\<x\>^{-N}(-\Delta+1)h_j}_{L^{2}}=0$ which, together with \eqref{proof_compactness_5}, shows that there exists $j_\ep\in \N$ such that, for all $j\ge j_\ep$, $\norm{h_j}_{L^{q_s,2}(\R^n)}<\ep$. Since $\ep>0$ is arbitrarily small, this shows that $h_j\to 0$ strongly in $L^{q_s,2}$ and \eqref{proof_compactness_4} follows. 

Finally, we shall show $R_0^+(z)w$ is continuous on $\overline{\C^+}$ in the operator norm topology of $\mathbb B(L^{q_s,2})$. Assume for contradiction that this is not the case. Then there exist $z_j,z\in\overline {\C^+}$ with $z_j\to z$ and $g_j\in L^{q_s,2}$ with $\norm{g_j}_{L^{q_s,2}}\le1$ such that $\liminf\limits_{j\to\infty}\norm{(R_0^+(z_j)w-R_0^+(z)w)g_j}_{L^{q_s,2}}>0$. Extracting a subsequence if necessary we may assume $g_j\to g$ with some $g\in L^{q_s}$ weakly in $L^{q_s}$. Then, by the argument as above and the compactness of $R_0^+(z)w$, we have
$
\lim\limits_{j\to\infty} R_0^+(z_j)wg_j=R_0^+(z)wg=\lim\limits_{j\to\infty}R_0^+(z)wg_j,
$
which gives a contradiction, proving the desired assertion. 
\end{proof}

\subsection{The exceptional set}
\label{subsection_resonance}
Having Proposition \ref{compactness} in mind, we define the exceptional set of $H$ as follows. 
\begin{definition}				
\label{definition_resonance}
We say that $\lambda\in \mathcal E$ if there exist $1/2<s<3/2$ and $f\in L^{q_s,2}(\R^n)\setminus\{0\}$ such that $f=-R_0(\lambda)V f$, 
where $q_s=2n/(n-2s)$ and $R_0(\lambda)$ is replaced by $R_0(\lambda+i0)$ if $\lambda\ge 0$. $\mathcal E$ is said to be the {\it exceptional set} of $H$. $z\in \mathcal E\setminus\sigma_{\mathop{\mathrm{p}}}(H)$ is called a {\it resonance} of $H$. For $\lambda\in \mathcal E$, we denote the family of corresponding solutions by $\mathcal N_s(\lambda)$: 
\begin{align*}
\mathcal N_s(\lambda)
&:=\{f\in L^{q_s,2}(\R^n)\setminus\{0\}\ |\ f=-R_0(\lambda)V f\},
\end{align*}
where $R_0(\lambda)$ is replaced by $R_0^+(\lambda)$ if $\lambda\ge 0$. 
\end{definition}
Note that, since $R_0(\lambda-i0)f=\overline{R_0(\lambda+i0)\overline f}$, one has
\begin{align}
\label{proof_proposition_exceptional_set_0}
\mathcal N_s(\lambda)=\{f\in L^{q_s,2}(\R^n)\setminus\{0\}\ |\ f=-R_0^-(\lambda)V f\},\quad \lambda\ge0. 
\end{align}
The next lemma collects some basic properties of $\mathcal E$. 

\begin{proposition}
\label{proposition_exceptional_set}$ $
\begin{itemize}
\item[{\rm(1)}] $\mathcal E\subset \sigma(H)$, $\sigma_{\mathrm{p}}(H)\subset \mathcal E$ and $\mathcal E\cap(-\infty,0)=\sigma_{\mathrm{d}}(H)$.  Moreover, $\mathcal N_s(\lambda)$ is finite dimensional. 
\item[{\rm (2)}] $\mathcal N_s(\lambda)$ is independent of $1/2<s<3/2$; that is, $\mathcal N_s(\lambda)=\mathcal N_{s'}(\lambda)$ for any $1/2<s,s'<3/2$. 
\end{itemize}
\end{proposition}

\begin{proof}[Proof of Proposition \ref{proposition_exceptional_set} (1)]
To prove  $\mathcal E\subset \sigma(H)$, we first claim that 
\begin{align}
\label{proof_proposition_exceptional_set_1}
\mathcal N_s(\lambda)=\{f\in \dot\H^s\ |\ f=-R_0(\lambda)Vf\},\quad \lambda\in \C\setminus(0,\infty).
\end{align}
Indeed, if we set $\wtilde{\mathcal N}_s(\lambda):=\{f\in \dot\H^s\ |\ f=-R_0(\lambda)Vf\}$ then the inclusion $\wtilde{\mathcal N}_s(\lambda)\subset \mathcal N_s(\lambda)$ is obvious since $\dot \H^s\subset L^{q_s,2}$ by the HLS inequality \eqref{HLS}. On the other hand, the HLS inequality \eqref{HLS} shows that $R_0(\lambda)V\in \mathbb B(L^{q_s,2},\dot\H^s)$ for $\lambda\in \C\setminus(0,\infty)$ and  the opposite inclusion $\wtilde{\mathcal N}(\lambda)_s\supset \mathcal N_s(\lambda)$ thus holds. Next, we let $f\in \mathcal N_s(\lambda)$ with some $\lambda\in \C\setminus\sigma(H)$. Then $Vf\in \dot \H^{2-s}\cap L^{p_s,2}$ by the HLS and H\"older's inequalities for Lorentz norms. Therefore, by Corollary \ref{corollary_free_Sobolev_1} (3), $(-\Delta-\lambda )f=-Vf$ holds in the distribution sense. In particular, 
$
\lambda f=(-\Delta+V)f\in \dot \H^{2-s}\cap \dot\H^s\subset L^2
$ and thus $f\in D(H)$. Since $\sigma(H)\subset\R$, this shows $f\equiv0$. Therefore, we obtain $\mathcal E\subset \sigma(H)$. 

The inclusion $\sigma_{\mathrm{p}}(H)\subset \mathcal E$ is obvious since $D(H)\subset \H^1\subset \dot \H^1$. This inclusion, together with the fact $\sigma(H)\cap(-\infty,0)=\sigma_{\mathrm{d}}(H)$, implies $\mathcal E\cap(-\infty,0)=\sigma_{\mathrm{d}}(H)$. Finally, since $R_0^\pm(z)V$ are compact operators on $L^{q_s,2}$, one has $\dim \mathcal N_s(\lambda)<\infty$. 
\end{proof}

To prove the second part of Proposition \ref{proposition_exceptional_set}, we need the following 

\begin{lemma}
\label{lemma_equivalence_1} 
For $1/2<s<3/2$ and real-valued functions $V_1\in L^{n/s,\infty}_0,V_2\in L^{n/(2-s),\infty}_0$ with $V=V_1V_2$, we set $K_s^+(\lambda):=V_1R^+(\lambda)V_2$. Then, for $\lambda\in\R$, 
\begin{align*}
\dim \mathcal N_s(\lambda)=\dim\Ker(I+K_s^+(\lambda))&=\dim\Ker(I+K_{s}^+(\lambda)^*)=\dim \mathcal N_{2-s}(\lambda). 
\end{align*}
\end{lemma}

\begin{remark}
Such $V_1,V_2$ always exist. Indeed, one can take $V_1=|V|^{\frac s2}$ and $V_2=\sgn V|V|^{\frac{2-s}{2}}$. 
\end{remark}

\begin{proof}
H\"older's inequality \eqref{Holder} and \eqref{corollary_free_Sobolev_1_2} yield that
\begin{align*}
\norm{V_1f}_{L^2}&\le C\norm{V}_{L^{\frac ns,\infty}}\norm{f}_{L^{q_s,2}},\quad
\norm{R_0^\pm(\lambda)V_2u}_{L^{q_s,2}}
\lesssim\norm{V_2}_{L^{\frac{n}{2-s},\infty}}\norm{u}_{L^2},
\end{align*}
from which one has two continuous maps
\begin{align*}
\mathcal N_s(\lambda)\ni f\mapsto V_1f\in \Ker(I+K_s^+(\lambda)),\quad
\Ker(I+K_s^+(\lambda))\ni u\mapsto -R_0^+(\lambda)V_2u\in \mathcal N_s(\lambda). 
\end{align*}
Furthermore, one also has, for $f\in \mathcal N_s(\lambda)$ and $u\in \Ker(I+K_s(\lambda))$, 
$$
-R_0^+(\lambda)V_2V_1f=-R_0^+(\lambda)V f=f,\quad -V_1R_0^+(\lambda)V_2u=u.
$$ 
Therefore, the multiplication by $V_1$ is a bijection between $\mathcal N_s(\lambda)$ and $\Ker(I+K_s^+(\lambda))$ and its inverse is given by $-R_0^+(\lambda)V_2$. In particular, $\dim\Ker(I+K_s^+(\lambda))=\dim \mathcal N_s(\lambda)$. 

Taking the facts $R_0^\pm(z)^*=R_0^\mp(\overline z)$ and \eqref{proof_proposition_exceptional_set_0} into account, it can be seen from the same argument that the multiplication by $V_2$ is a bijection between $\mathcal N_{2-s}(\lambda)$ and $\Ker(I+K_{s}^+(\lambda)^*)$, and its inverse is given by $-R_0^-(\lambda)V_1$. In particular, $\dim \mathcal N_{2-s}(\lambda)=\dim\Ker(I+K_{s}^+(\lambda)^*)$. 

For the part $\dim\Ker(I+K_s^+(\lambda))=\dim\Ker(I+K_{s}^+(\lambda)^*)$, since $K_s^+(\lambda)$ is compact on $L^2$ (see Corollary \ref{corollary_free_Sobolev_2}), $I+K_s^+(\lambda)$ is Fredholm and its index satisfies 
$$
\dim\Ker(I+K_s^+(\lambda))-\mathrm{codim}\Ran(I+K_s^+(\lambda))=\mathrm{ind}(I+K_s^+(\lambda))=\mathrm{ind}I=0. 
$$ 
Therefore, taking the fact $L^2/\Ran(I+K_s^+(\lambda))\cong [\Ran(I+K_s^+(\lambda))]^\perp$ into account, one has
\begin{align*}
\dim\Ker(I+K_s^+(\lambda))&=\dim [\Ran(I+K_s^+(\lambda))]^\perp=\dim\Ker(I+K_s^{+}(\lambda)^*),
\end{align*}
which completes the proof. 
\end{proof}

\begin{proof}[Proof of Proposition \ref{proposition_exceptional_set} (2)]
Let $f\in \mathcal N_s(\lambda)$ and $1/2<s\le s'<3/2$. Let $V=v_1+v_2$ be such that $v_1\in C_0^\infty$ and $\norm{v_2}_{L^{n/2,\infty}}\le \ep$. Then $f=-R_0^+(\lambda)v_1f-R_0^+(\lambda)v_2f$. By Proposition \ref{compactness}, the map $I+R_0^+(\lambda)v_2:L^{\frac{2n}{n-2r},2}\to L^{\frac{2n}{n-2r},2}$ is bounded and invertible  for $r=s,s'$ and small $\ep>0$. If $E_r$ denotes the inverse of $I+R_0^+(\lambda)v_2:L^{\frac{2n}{n-2r},2}\to L^{\frac{2n}{n-2r},2}$, then $E_s=E_{s'}$ on $L^{\frac{2n}{n-2s},2}\cap L^{\frac{2n}{n-2s'},2}$. Taking the inequality $s-s'>-1$ into account, the HLS inequality \eqref{HLS} implies
\begin{align*}
\norm{R_0^+(\lambda)v_1f}_{L^{\frac{2n}{n-2s'},2}}
\lesssim \norm{v_1f}_{L^{\frac{2n}{n+2(2-s')}}}\lesssim \norm{v_1}_{L^\frac{n}{2+2(s-s')}}\norm{f}_{L^{\frac{2n}{n-2s}}}.  
\end{align*}
Thus $R_0^+(\lambda)v_1f\in L^{\frac{2n}{n-2s},2}\cap L^{\frac{2n}{n-2s'},2}$
 and $f=E_sR_0^+(\lambda)v_1f=E_{s'}R_0^+(\lambda)v_1f\in L^{\frac{2n}{n-2s'},2}$, which implies $f\in \mathcal N_{s'}(\lambda)$. Therefore $\mathcal N_s(\lambda)$ is monotonically increasing in $s$. Combining with the fact $\dim \mathcal N_s(\lambda)=\dim\mathcal N_{2-s}(\lambda)<\infty$ (see Lemma \ref{lemma_equivalence_1}), this monotonicity implies $\mathcal N_s(\lambda)=\mathcal N_{s'}(\lambda)$. 
\end{proof}

We conclude this subsection to prove Lemma \ref{lemma_absence_1}. For the first part, we employ the following results by Ionescu-Jerison \cite{IoJe} and by Ionescu-Schlag \cite{IoSc}. 

\begin{proposition}[{\cite[Theorem 2.1]{IoJe}}]
\label{proposition_IoJe}
Let $n\ge3$ and $V\in L^{n/2}$. Suppose that $f\in \H^1_{\mathrm{loc}}$ and $\<x\>^{-1/2+\delta}f\in L^2$ with some $\delta>0$. If $-\Delta f+Vf=\lambda f$ for some $\lambda>0$, then $f\equiv0$. 
\end{proposition}

Let us set $X=\mathcal W^{-\frac{1}{n+1},\frac{2(n+1)}{n+3}}+S_1(B)$, where  $B$ is the Agmon-H\"ormander space  and $S_{1}(B)$ is the image of $B$ under $S_{1}=(1-\Delta)^{1/2}$ (see \cite{IoSc}). Then $X^*=\mathcal W^{\frac{1}{n+1},\frac{2(n+1)}{n-1}}\cap S_{-1}(B^*)$ and we have the continuous embeddings $L^{\frac{2n}{n+2}}\subset X$ and $X^*\subset L^{\frac{2n}{n-2}}$. Moreover, it was proved in \cite[Lemma 4.1 (b)]{IoSc} that $R_0^\pm(\lambda)\in \mathbb B(X,X^*)$ for all $\lambda\in \R\setminus\{0\}$. 

\begin{proposition}[{\cite[Lemma 4.4]{IoSc}}]
\label{proposition_IoSc}
Let $n\ge3$ and $V\in L^{n/2}$. Assume that $f$ belongs to $X^*$ and satisfies $f+R_0^\pm(\lambda)Vf=0$ for some $\lambda\in \R\setminus\{0\}$. Then, for any $N\ge0$, 
$$\norm{\<x\>^Nf}_{X^*}\le C_{N,\lambda}\norm{f}_{X^*}.
$$
\end{proposition}

\begin{proof}[Proof of Lemma \ref{lemma_absence_1}] For the proof of the part (1), we let $f\in \mathcal N_1(\lambda)$ with $\lambda>0$. As observed in the proof of Proposition \ref{compactness}, $R_0^+(\lambda)V$ maps from $L^{\frac{2n}{n-2}}(\R^n)$ into $\mathcal W^{2,\frac{2n}{n-2}}(\R^n)$ (see \eqref{proof_compactness_2}) and thus $f=-R_0^+(\lambda)V f\in \H^1_{\loc}$. Moreover, since $Vf\in L^{\frac{2n}{n+2}}\subset X$ and $R_0^\pm(\lambda)\in \mathbb B(X,X^*)$, we have  $f\in X^*$. Proposition \ref{proposition_IoSc} then implies that $f\in L^2$. 
Using Proposition \ref{proposition_IoJe}, we conclude that $f\equiv0$. For the part (2), we let $f\in \mathcal N_1(0)$. Since $-\Delta f+Vf\in \dot \H^{-1}$, the form $\<-\Delta f+Vf,f\>$ is well-defined. By assumption, we have $0=\<-\Delta f+Vf,f\>\ge \delta\norm{f}_{\dot\H^1}$ which implies $f\equiv0$. 
\end{proof}

\section{Uniform Sobolev estimates}
\label{section_uniform_Sobolev}
This section is devoted to the proof of Theorem \ref{theorem_KRS_1}, Corollaries \ref{corollary_KRS_1} and \ref{corollary_KRS_2} and Theorem \ref{theorem_spectral_measure_1}. We begin with the following proposition which  plays an important role in the proof. 

\begin{proposition}	
\label{proposition_weighted_0}
Assume  $1/2<s<3/2$ and let $(p_s,q_s)$ be as in \eqref{p_q_s}. Then $(I+R_0^\pm(z)V)^{-1}$ are $\mathbb B(L^{q_s,2})$-valued continuous functions on $\overline{\C^\pm}\setminus\mathcal E$, respectively. Furthermore, for any $\delta>0$, 
 \begin{align}
\label{proposition_weighted_0_1}
\sup_{z\in\overline{\C^\pm}\setminus\mathcal E_\delta}\norm{(I+R_0^\pm(z)V)^{-1}}_{\mathbb B(L^{q_s,2})}<\infty. 
\end{align}
In particular, if $\mathcal E\cap[0,\infty)=\emptyset$, then $\sup\limits_{z\in \C\setminus\R}\norm{(I+R_0(z)V)^{-1}}_{\mathbb B(L^{q_s,2})}<\infty$. 
\end{proposition}

The proof of Proposition \ref{proposition_weighted_0} is divided into a series of lemmas. Let us prove the proposition for $z\in \overline{\C^+}\setminus\mathcal E$ only, the proof for the case $z\in \overline{\C^-}\setminus\mathcal E$ being analogous. 


\begin{lemma}
\label{lemma_weighted_2}
$(I+R_0^+(z)V)^{-1}$ is a $\mathbb B(L^{q_s,2})$-valued continuous function on $\overline{\C^+}\setminus\mathcal E$. 
\end{lemma}
\begin{proof}
By Proposition \ref{compactness},  $R_0^+(z)V$ is compact. 
Since $\mathcal N_s(z)=\{0\}$ for $z\in  \overline{\C^+}\setminus\mathcal E$ by definition, 
the Fredholm alternative ensures the existence of $(I+R_0^+(z)V)^{-1}\in \mathbb B(L^{q_s,2})$. Moreover, since $R_0^+(z)V$ is continuous on $\overline{\C^+}$ in the operator norm topology of $\mathbb B(L^{q_s,2})$ by Proposition \ref{compactness}, $(I+R_0^+(z)V)^{-1}$ is also continuous on $\overline{\C^+}\setminus \mathcal E$ in the same topology. 
\end{proof}

The proof of  the uniform bound \eqref{proposition_weighted_0_1} is divided into high, intermediate and low energy parts. 
\begin{lemma}[The high energy estimate]
\label{lemma_weighted_3}
There exists $L\ge1$ such that $(I+R_0^+(z)V)^{-1}$ is bounded on $L^{\frac{2n}{n-2s},2}$ uniformly in $z\in \overline{\C^+}\cap\{|z|\ge L\}$. 
\end{lemma}

\begin{proof}
Let $V_k\in C_0^\infty(\R^n)$ be such that $\lim\limits_{k\to\infty}\norm{V-V_k}_{L^{\frac n2,\infty}}=0$ and  set $Q_k^+(z):=R_0^+(z)(V-V_k)$. By Proposition \ref{corollary_free_Sobolev_1} with 
$
(p_s,q_s)
$, one can find $k_0\ge1$ such that 
$$
\sup\limits_{z\in \overline{\C^+}}\norm{Q_{k_0}^+(z)}_{\mathbb B(L^{q_s,2})} \le 1/2. 
$$
Hence $(I+Q_{k_0}(z))^{-1}$ is defined by the Neumann series $\sum\limits_{n=0}^\infty (-Q_{k_0}^+(z))^{n}$ and satisfies $$M_1:=\sup\limits_{z\in \overline{\C^+}}\norm{(I+Q^+_{k_0}(z))^{-1}}_{\mathbb B(L^{q_s,2})} \le 2. $$ 
Next if we take $p_\delta$ and small $\delta>0$ such that ${1}/{p_\delta}=1/p_s-\delta$ and $(p_\delta,q_s)$ satisfies \eqref{p_q}, Proposition \ref{corollary_free_Sobolev_1} implies
$$
\norm{R_0^+(z)V_{k_0}f}_{L^{q_s,2}} \lesssim |z|^{-\delta}\norm{V_{k_0}f}_{L^{p_\delta,2}}\lesssim |z|^{-\delta}\norm{V_{k_0}}_{L^r}\norm{f}_{L^{q_s,2}}
$$
uniformly in $|z|\ge1$ and $f\in L^{q_s,2}$, where $1/r=1/p_\delta-1/q_s=2/n-\delta$. Hence one can find $L=L_{k_0}$ so large that $M_2:=\norm{R_0^+(z)V_{k_0}}_{\mathbb B(L^{q_s,2})} \le 1/4$ for $|z|\ge L$. 
Then, writing 
$$
I+R_0^+(z)V=I+Q^+_{k_0}(z)+R_0^+(z)V_{k_0}=(I+Q^+_{k_0}(z))\Big(I+(I+Q^+_{k_0}(z))^{-1}R_0^+(z)V_{k_0}\Big),
$$
we see that $(I+R_0^+(z)V)^{-1}=\Big(I+(I+Q^+_{k_0}(z))^{-1}R_0^+(z)V_{k_0}\Big)^{-1}(I+Q^+_{k_0}(z))^{-1}$ and
$$
\sup_{z\in \overline{\C^+}\cap\{|z|\ge L\}}\norm{(I+R_0^+(z)V)^{-1}}_{\mathbb B(L^{q_s,2})} \le M_1\sum_{n=1}^\infty (M_1M_2)^n\le 4. 
$$
This completes the proof. 
\end{proof}

\begin{remark}
\label{remark_weighted_4}
This lemma particularly implies $\mathcal E\cap[L,\infty)=\emptyset$ and thus $\mathcal E$ is bounded in $\R$. 
\end{remark}

\begin{lemma}[The intermediate energy estimate]
\label{lemma_weighted_5}
For any $\delta,L>0$, $(I+R_0^+(z)V)^{-1}$ is bounded on $L^{q_s,2}$ uniformly in $z\in (\overline{\C^+}\setminus\mathcal E_{\delta})\cap\{\delta<|z|<L\}$. 
\end{lemma}

\begin{proof}
We follow the argument in \cite[Lemma 4.6]{IoSc} closely. Let $\Lambda_{\delta,L}=(\overline{\C^+}\setminus\mathcal E_{\delta})\cap\{\delta<|z|<L\}$. Note that $\overline{\Lambda_{\delta,L}}\cap \mathcal E=\emptyset$. Assume for contradiction that 
$$
\sup_{z\in \Lambda_{\delta,L}}\norm{(I+R_0^+(z)V)^{-1}}_{\mathbb B(L^{q_s,2})} =\infty.
$$
Then one can find $f_j\in L^{q_s,2}$ with $\norm{f_j}_{L^{q_s,2}}=1$ and $z_j\in \Lambda_{\delta,L}$ such that
\begin{align}
\label{proof_lemma_weighted_5_1}
\norm{(I+R_0^+(z_j)V)f_j}_{\mathbb B(L^{q_s,2})}\to 0,\quad j\to\infty. 
\end{align}
By passing to a subsequence, we may assume $z_j\to z_\infty\in \overline{\Lambda_{\delta,L}}$ as $j\to \infty$. Since $R_0^+(z_\infty)V$ is compact on $L^{q_s,2}$, by passing to a subsequence, we may assume without loss of generality that there exists $g\in L^{q_s,2}$ such that $R_0^+(z_\infty)Vf_j\to g$ strongly in $L^{q_s,2}$. By virtue of \eqref{proof_lemma_weighted_5_1} and the condition $\norm{f_j}_{L^{q_s,2}}=1$, we have $g\not\equiv0$. Now we claim that $g$ belongs to $\mathcal N_s(z_\infty)$, which implies $z_\infty\in \mathcal E$. This contradicts with $z_\infty \in \overline{\Lambda_{\delta,L}}$. 

In order to prove the claim, we write $f_j$ as 
$$f_j=\Big(I+R_0^+(z_j)V\Big)f_j-\Big(R_0^+(z_j)-R_0^+(z_\infty)\Big)Vf_j-R_0^+(z_\infty)Vf_j$$
By virtue of \eqref{proof_lemma_weighted_5_1} and the continuity of $R^+_0(z)V$ (see Proposition \ref{compactness}) and the fact $\norm{f_j}_{L^{q_s,2}}=1$, the right hand side converges to $-g$  strongly in $L^{q_s,2}$ as $j\to\infty$. 
Therefore, we have $g=-R_0^+(z_\infty)V g$. Moreover, since $\norm{f_j}=1$, $g\not\equiv0$ and hence $g\in \mathcal N_s(z_\infty)$ follows. \end{proof}

Lemmas \ref{lemma_weighted_3} and \ref{lemma_weighted_5} give the desired bound \eqref{proposition_weighted_0_1} for the case when $0\in \mathcal E$. When $0\notin \mathcal E$, we need the following lemma to complete the proof of Proposition \ref{proposition_weighted_0}. 

\begin{lemma}[The low energy estimate]
\label{lemma_weighted_4}
Suppose that $0\notin \mathcal E$. Then there exists $\delta>0$ such that $(I+R_0^+(z)V)^{-1}$ is bounded on $L^{q_s,2}$ uniformly in $z\in \overline{\C^+}\cap\{|z|\le\delta\}$. 
\end{lemma}

\begin{proof}
Since $I+R_0^+(0)V$ is invertible if $0\notin \mathcal E$ by Lemma \ref{lemma_weighted_2}, one can write
$$
I+R_0^+(z)V
=(I+R_0^+(0)V)\Big(I+(I+R_0^+(0)V)^{-1}\big(R_0^+(z)-R_0^+(0)\big)V\Big). 
$$
Since $\overline{\C^+}\ni z\mapsto R_0^+(z)V\in \mathbb B(L^{q_s,2})$ is continuous by Proposition \ref{compactness}, one has 
$$\sup_{z\in \overline{\C^+}\cap\{|z|\le\delta\}}\norm{\big(R_0^+(z)-R_0^+(0)\big)V}_{\mathbb B(L^{q_s,2})}\le \frac{1}{2\norm{(I+R_0^+(0)V)^{-1}}}$$ for $\delta>0$ small enough. Therefore, 
$I+R_0^+(z)V$ is invertible on $L^{q_s,2}$ and
$$\sup_{z\in \overline{\C^+}\cap\{|z|\le\delta\}}\norm{(I+R_0^+(z)V)^{-1}}_{\mathbb B(L^{q_s,2})} \le 2\sup_{z\in \overline{\C^+}\cap\{|z|\le\delta\}}\norm{(I+R_0^+(0)V)^{-1}}_{\mathbb B(L^{q_s,2})} <\infty
$$
which completes the proof.
\end{proof}

By Lemmas \ref{lemma_weighted_2}--\ref{lemma_weighted_5}, we have completed the proof of Proposition \ref{proposition_weighted_0}. 

We next give a rigorous justification of the second resolvent equation. 

\begin{lemma}
\label{lemma_resolvent_identity_1}
Let $z\in \C\setminus\sigma(H)$.  Then, as a bounded operator from $L^2$ to $D(H)$, 
\begin{align}
\label{lemma_resolvent_identity_1_1}
R(z)=(I+R_0(z)V)^{-1}R_0(z)=R_0(z)-R_0(z)VR(z). 
\end{align}
Moreover, we also obtain for $z,z'\in \C\setminus\sigma(H)$, 
\begin{align}
\label{lemma_resolvent_identity_1_2}
R(z)-R(z')=(I+R_0(z')V)^{-1}(R_0(z)-R_0(z'))(I-VR(z)).
\end{align}
\end{lemma}

\begin{proof}
It follows from Proposition \ref{proposition_exceptional_set} (1) and the fact $\H^1\subset L^{\frac{2n}{n-2},2}$ that $\Ker_{\H^1}(I+R_0(z)V)$ is trivial. 
Since $R_0(z)V\in \mathbb B_\infty(\H^1)$ by Proposition \ref{compactness}, $I+R_0(z)V$ is invertible on $\H^1$ by the Fredholm alternative theorem. $(I+R_0(z)V)^{-1}R_0(z)$ thus is a bounded operator from $L^2$ to $\H^1$. Let $f\in L^2$ and set $g=(I+R_0(z)V)^{-1}R_0(z)f\in \H^1$. Since $$(I+R_0(z)V)(I+R_0(z)V)^{-1}R_0(z)=R_0(z)$$
as a bounded operator from $L^2$ to $\H^1$, we see that
\begin{align}
\label{proof_lemma_resolvent_identity_1_1}
g=R_0(z)f-R_0(z)Vg.
\end{align} 
Then, for any $\varphi\in \H^1$, 
$
\<(-\Delta-z)g,\varphi\>=\<f,\varphi\>-\<Vg,\varphi\>=\<f,\varphi\>-\<V_1g,V_2\varphi\>,
$
where $V_1,V_2\in L^{n/2,\infty}_0(\R^n;\R)$ satisfies $V=V_1V_2$. Therefore, we obtain
\begin{align*}
\<(H-z)g,\varphi\>=\<(-\Delta-z)g,\varphi\>+\<V_1g,V_2\varphi\>=\<f,\varphi\>
\end{align*}
which shows $(H-z)(I+R_0(z)V)^{-1}R_0(z)=I$ on $L^2$. For $f\in D(H)$, we similarly obtain
$$
(I+R_0(z)V)^{-1}R_0(z)(H-z)f=(I+R_0(z)V)^{-1}f+(I+R_0(z)V)^{-1}R_0(z)Vf=f,
$$
which gives us $(I+R_0(z)V)^{-1}R_0(z)(H-z)=I$ on $D(H)$ and the first identity in \eqref{lemma_resolvent_identity_1_1} thus follows. The second identity in \eqref{lemma_resolvent_identity_1_1} follows from the first identity and \eqref{proof_lemma_resolvent_identity_1_1}. 

Now we shall show \eqref{lemma_resolvent_identity_1_2}. It follows from  \eqref{lemma_resolvent_identity_1_1} that
$$
(I+R_0(z')V)(R(z)-R(z'))=(R_0(z)-R_0(z'))(I-VR(z))
$$
on $L^2$. Since $R_0(z)-R_0(z'),R(z)-R(z'):L^2\to \H^1$ are continuous and $I+R_0(z')V$ is invertible on $\H^1$, we have the desired identity \eqref{lemma_resolvent_identity_1_2}. 
\end{proof}

Now we are in position to prove Theorem \ref{theorem_KRS_1}, Corollaries \ref{corollary_KRS_1} and \ref{corollary_KRS_2} and  Theorem \ref{theorem_spectral_measure_1}. 

\begin{proof}[Proof of Theorem \ref{theorem_KRS_1}]
Assume that $(p,q)$ satisfies \eqref{p_q}. It follows from Propositions \ref{proposition_free_Sobolev_1} and \ref{proposition_weighted_0} and Lemma \ref{lemma_resolvent_identity_1} that, for any $\delta>0$ there exists $C_\delta>0$ such that 
\begin{align*}
\norm{R(z)f}_{L^{q,2}}\le C_\delta (1+\norm{(I+R_0(z)V)^{-1}}_{\mathbb B(L^{q,2})})\norm{R_0(z)f}_{L^{q,2}}\le C_\delta |z|^{\frac n2(\frac1p-\frac1q)-1}\norm{f}_{L^{p,2}}
\end{align*}
for all $f\in L^2\cap L^{p,2}$ and $z\in \C\setminus([0,\infty)\cup \mathcal E_\delta)$. Since $L^2\cap L^{p,2}$ is dense in $L^{p,2}$, this implies that $R(z)\in \mathbb B(L^{p,2},L^{q,2})$ and that \eqref{theorem_KRS_1_1} holds uniformly in $z\in \C\setminus([0,\infty)\cup \mathcal E_\delta)$. 
\end{proof}

\begin{proof}[Proof of Corollary \ref{corollary_KRS_1}]
As before, we shall prove the corollary for $R(\lambda+i0)$ only. We also consider the case $1/p-1/q=2/n$ only, proof for other cases being similar. At first, we claim that, for any $\chi_1,\chi_2\in C_0^\infty(\R^n)$, $\chi_1R(z)\chi_2$ defined for $z\in \C^+$ extends to a $\mathbb B(L^2)$-valued continuous function $\chi_1R^+(z)\chi_2$ on $\overline{\C^+}\setminus\mathcal E$. It follows from this claim that, for any $u,v\in C_0^\infty(\R^n)$, $\<R^+(z)u,v\>$ is a continuous function on $\overline{\C^+}\setminus\mathcal E$. Then, by letting $\ep\searrow0$ in the estimate
$$
|\<R(\lambda+i\ep)u,v\>|\lesssim \norm{u}_{L^{p,2}}\norm{v}_{L^{q',2}},
$$
which follows from Theorem \ref{theorem_KRS_1}, and by using the density argument we obtain  that $R(\lambda+i0)$ extends to an element in $\mathbb B(L^{p,2},L^{q,2})$ and satisfies
\begin{align}
\label{proof_corollary_KRS_1_1}
\sup_{\lambda\in [0,\infty)\setminus\mathcal E}\norm{R(\lambda+i0)}_{\mathbb B(L^{p,2},L^{q,2})}<\infty. 
\end{align}
 This shows the first statement (1). For the second statement (2), it follows by plugging $z=\lambda\pm i\ep$ and then letting $\ep\searrow0$ in the equation \eqref{lemma_resolvent_identity_1_1} that, for any $f\in L^{q,2}\cap L^2$ and $\lambda\in [0,\infty)\setminus\mathcal E$, 
\begin{align}
\label{proof_corollary_KRS_1_2}
R(\lambda\pm i0)f=R_0(\lambda\pm i0)\Big(I-VR(\lambda\pm i0)\Big)f
\end{align}
in the sense of distributions, which particularly implies that, under the condition $0\notin \mathcal E$, $R(0+i0)=R(0-i0)$ since $R_0(0\pm i0)=(-\Delta)^{-1}$. Moreover, we also learn by \eqref{proof_corollary_KRS_1_2} that 
\begin{align*}
(-\Delta+V-\lambda)R(\lambda+i0)u
&=(I+VR_0(\lambda+i0))(I-VR(\lambda+i0))u\\
&=u+V[R_0(\lambda+i0)-R(\lambda+i0)-R_0(\lambda+i0)VR(\lambda+i0)]u=u
\end{align*} 
for all $u\in L^2\cap L^{p,2}$ and that, for all $v\in \S$, 
\begin{align*}
R(\lambda+i0)(-\Delta+V-\lambda)v
&=R_0(\lambda+ i0)\Big(I-VR(\lambda+ i0)\Big)(-\Delta+V-\lambda)v\\
&=v-R_0(\lambda+ i0)Vv-R_0(\lambda+ i0)Vv=v
\end{align*}
in the sense of distributions. These two identities and \eqref{proof_corollary_KRS_1_1} imply \eqref{corollary_KRS_1_2}. 

It remains to show the above claim. 
Let $V_1,V_2\in L^{n,\infty}_0(\R^n;\R)$ be such that $V=V_1V_2$ and set $K_1(z)=V_1R_0(z)V_2$. The resolvent identity \eqref{lemma_resolvent_identity_1_1} then yields
$$
V_1R(z)\chi_2=V_1R_0(z)\chi_2-K_1(z)V_1R(z)\chi_2
$$
on $L^2$ for all $z\in \C\setminus\sigma(H)$. 
Since $K_1(z)\in \mathbb B_\infty(L^2)$ by Corollary \ref{corollary_free_Sobolev_2} and $\Ker_{L^2}(I+K_1(z))=\emptyset$ for all $z\in \C\setminus\sigma(H)$ by Proposition \ref{proposition_exceptional_set} and Lemma \ref{lemma_equivalence_1}, we learn by this identity that
\begin{align*}
V_1R(z)\chi_2=(I+K_1(z))^{-1}V_1R_0(z)\chi_2,\quad z\in \C\setminus\sigma(H),
\end{align*}
on $L^2$. It follows from again Corollary \ref{corollary_free_Sobolev_2} that $V_1R_0(z)\chi_2$ and $K_1(z)$ extend to $\mathbb B_\infty(L^2)$-valued continuous functions $V_1R_0^+(z)\chi_2$ and $K_1^+(z)=V_1R_0^+(z)V_2$ on $\overline{\C^+}$. Since $\Ker(I+K_1^+(z))=\emptyset$ for $z\in \overline{\C^+}\setminus \mathcal E$, $(I+K_1(z))^{-1}$ also extends to a $\mathbb B(L^2)$-valued continuous function $(I+K_1^+(z))^{-1}$ on $\overline{\C^+}\setminus\mathcal E$. $V_1R(z)\chi_2$ thus extends to a $\mathbb B(L^2)$-valued continuous function $V_1R^+(z)\chi_2$ on $\overline{\C^+}\setminus\mathcal E$, satisfying
$
V_1R^+(z)\chi_2=(I+K_1^+(z))^{-1}V_1R_0^+(z). 
$
Finally, the claim follows from the formula
$$
\chi_1R(z)\chi_2=\chi_1 R_0(z)\chi_2-\chi_1R_0(z)V_2V_1R(z)\chi_2
$$
and the continuity of $\chi_1 R_0^+(z)\chi_2$, $\chi_1R_0^+(z)V_2$ and $V_1 R_0^+(z)\chi_2$ on $\overline{\C^+}\setminus \mathcal E$. 
\end{proof}

\begin{proof}[Proof of Corollary \ref{corollary_KRS_2}]
Let us fix $z\in \C\setminus \sigma(H)$ and take $\delta>0$ so small that $z\notin \mathcal E_\delta$. Recall that $R_0(z)\in \mathbb B(L^{p})$ for all $1\le p\le \infty$ and thus $R_0(z)\in \mathbb B(L^{p,2})$ for all $1<p<\infty$ by Theorem \ref{theorem_interpolation_2}. 

The proof of the first assertion is divided into two cases: $\frac{2n}{n+3}<p=q<\frac{2n}{n+1}$ and otherwise. Firstly, when $\frac{2n}{n+3}<p=q<\frac{2n}{n+1}$,  one can find $\frac{2n}{n-1}<q_0<\frac{2n}{n-3}$ such that $\frac1p-\frac{1}{q_0}=\frac2n$. Applying Theorem \ref{theorem_KRS_1} to the resolvent equation \eqref{lemma_resolvent_identity_1_1} implies that, for all $f\in L^2\cap L^{p,2}$, 
\begin{align*}
\norm{R(z)f}_{L^{p,2}}\lesssim \norm{R_0(z)f}_{L^{p,2}}+\norm{R_0(z)}_{\mathbb B(L^{p,2})}\norm{V}_{L^{\frac n2,\infty}}\norm{R(z)f}_{L^{q_0,2}}\le C_\delta\norm{f}_{L^{p,2}}.
\end{align*}
 Combined with a density argument, this implies $R(z)\in \mathbb B(L^{p,2})$ for each $z\in \C\setminus \sigma(H)$. 

Next, by taking the adjoint and using the fact $R(z)^*=R(\overline z)$, we see that $R(z)\in \mathbb B(L^{p,2})$ for all $\frac{2n}{n-1}<p<\frac{2n}{n-3}$. Interpolating these two case yields that $R(z)\in \mathbb B(L^{p,2})$ for all $\frac{2n}{n+3}<p<\frac{2n}{n-3}$. Then the other cases in the first assertion follows by interpolating between the estimates on the two lines $\frac1p-\frac1q=0$ and $\frac1p-\frac1q=\frac2n$ under the conditions $\frac{2n}{n+3}<p$ and $q<\frac{2n}{n-3}$. 

Finally, assuming $1/2<s<3/2$ without loss of generality, the second assertion follows from
$$
\norm{wR(M)f}_{L^2}\lesssim \norm{w}_{L^{\frac ns,\infty}}\norm{R(M)f}_{L^{\frac{2n}{n-2s},2}}\lesssim \norm{w}_{L^{\frac ns,\infty}}\norm{f}_{L^2}
$$
for $M<\inf\sigma(H)-1$, which is a particular case of the first assertion. 
\end{proof}

\begin{proof}[Proof of Theorem \ref{theorem_spectral_measure_1}]
When $\frac{2n}{n+2}\le p\le \frac{2(n+1)}{n+3}$, \eqref{theorem_spectral_measure_1_1} follows from \eqref{corollary_KRS_1_1} and Stone's formula \eqref{Stone}. 
When $\frac{2n}{n+3}<p<\frac{2n}{n+2}$, there are two main ingredients. 

At first, it is known that $E'_{-\Delta}(\lambda)\in \mathbb B(L^p,L^{p'})$ for all $1\le p\le \frac{2(n+1)}{n+3}$ and satisfies
\begin{align}
\label{proof_theorem_spectral_measure_1_1}
\norm{E_{-\Delta}'(\lambda)}_{\mathbb B(L^p,L^{p'})}\lesssim\lambda^{\frac n2(\frac1p-\frac{1}{p'})-1},\quad \lambda>0.
\end{align}
Indeed, $E'_{-\Delta}(\lambda)$ can be brought to the form $E'_{-\Delta}(\lambda)=(2\pi)^{-n}\lambda^{(n-1)/2}R_{\sqrt{\lambda}}^*R_{\sqrt{\lambda}}$, where $$R_\mu u(\omega):=\int_{\R^n} e^{-2\pi i\mu\omega\cdot x}u(x)dx,\quad \mu>0,\ \omega\in \mathbb S^{n-1}.$$
Then the Stein-Tomas restriction theorem  (see \cite{Tom,Ste2}) and the $TT^*$-argument show that $R_1^*R_1$ is bounded from $L^p$ to $L^{p'}$ for all $1\le p\le \frac{2(n+1)}{n+3}$, which particularly implies \eqref{proof_theorem_spectral_measure_1_1} by scaling. 

Secondly, we claim that the following identity holds for all $f,g\in \S$ and $\lambda\in (0,\infty)$: 
\begin{align}
\label{proof_theorem_spectral_measure_1_2}
\<E'_H(\lambda)f,g\>=\<(I+R_0(\lambda-i0)V)^{-1}E'_{-\Delta}(\lambda)(I-VR(\lambda+i0))f,g\>.
\end{align}
Since $VR(\lambda+i0)\in \mathbb B(L^{p})$ and $(I+R_0(\lambda-i0)V)^{-1}\in \mathbb B(L^{p'})$ for $\frac{2n}{n+3}<p<\frac{2n}{n+1}$ by Corollary \ref{corollary_KRS_1} and Proposition \ref{proposition_weighted_0}, the desired assertion \eqref{theorem_spectral_measure_1_1} follows from \eqref{proof_theorem_spectral_measure_1_1}, \eqref{proof_theorem_spectral_measure_1_2} and a density argument. 

It remains to show the identity \eqref{proof_theorem_spectral_measure_1_2}. Let $f,g\in \S$ and set
$$
F(z)=\frac{1}{\pi}(I+R_0(\overline z)V)^{-1}\Im R_0(z)(I-VR(z)),\quad z\in \C^+, 
$$
which is a bounded operator from $L^2$ to $\H^1$ (see the proof of Lemma \ref{lemma_resolvent_identity_1}) where $\Im R_0(z)=(2i)^{-1}(R_0(z)-R_0(\overline z))$.  By \eqref{lemma_resolvent_identity_1_2} with $z=\lambda+i\ep$, $z'=\overline z$, one has 
$\pi^{-1}\Im R(z)=F(z)$. Moreover, $$\<E'_H(\lambda)f,g\>=\pi^{-1}\lim\limits_{\ep\searrow0}\<\Im R(\lambda+i\ep)f,g\>$$ exists by Corollary \ref{corollary_KRS_1}. For the operator $F(z)$,  we write
$$
F(z)f=\frac{1}{\pi}(I+R_0(\overline z)V)^{-1}(\Im R_0(z)\<x\>^{-3}-\Im R_0(z)VR(z)\<x\>^{-3})\<x\>^3f. 
$$
By Proposition \ref{compactness}, all of $(I+R_0(\overline z)V)^{-1},\Im R_0(z)\<x\>^{-3},\Im R_0(z)V$ and $R(z)\<x\>^{-3}$ extend to $\mathbb B(L^{p'})$-valued continuous function on $\overline {\C^+}\setminus\mathcal E$. Therefore, $\<F(\lambda+i0)f,g\>=\lim\limits_{\ep\searrow0}\<F(\lambda+i\ep)f,g\>$ exists and coincides with the right hand side of \eqref{proof_theorem_spectral_measure_1_2}. Therefore \eqref{proof_theorem_spectral_measure_1_2} follows. 
\end{proof}


The remaining part of the section is devoted to the following theorem, which plays a crucial role in the proof of Strichartz estimates.

\begin{theorem}
\label{theorem_KRS_2}
Suppose that $\mathcal E\cap [0,\infty)=\emptyset$. Let $(p,q)$ be such that $1/p-1/q=2/n$ and ${2n}/{(n+3)}<p<{2n}/{(n+1)}$. Then
\begin{align}
\label{theorem_KRS_2_2}
\sup_{z\in \C\setminus[0,\infty)}\norm{P_{\mathrm{ac}}(H) R(z)}_{\mathbb B(L^{p,2},L^{q,2})}<\infty. 
\end{align}
\end{theorem}

We first prove some $L^p$-boundedness of the projection $P_{\mathrm{ac}}(H)$. At first note that, under the condition $0\notin \mathcal E$, $H$ may have at most finitely many negative eigenvalues of finite multiplicities. Indeed, since $\sigma_{\mathrm{p}}(H)\cap(-\infty,0)=\sigma_{\mathrm{d}}(H)$, each negative eigenvalue has finite multiplicity and their only possible accumulation point is $z=0$. Moreover, Lemma \ref{lemma_weighted_4} and the Fredholm alternative show that, for sufficiently small $\delta>0$, $(-\delta,\delta)\cap \mathcal E=\emptyset$ as long as $0\notin \mathcal E$. Therefore, $H$ may have at most finitely many negative eigenvalues. In this case $P_{\mathrm{ac}}(H) $ is written in the form
\begin{align}
\label{AC}
P_{\mathrm{ac}}(H) =I-\sum_{j=1}^NP_j,\quad P_j:=\<\cdot,\psi_{j}\>\psi_{j}
\end{align}
where $\psi_j$ are eigenfunctions of $H$ and $N<\infty$. 

\begin{lemma}
\label{lemma_KRS_4}
$\psi_{j}\in L^{q,2}$ and $P_{\mathrm{ac}}(H)\in\mathbb B(L^{q,2})$ for all $\frac{2n}{n+3}< q< \frac{2n}{n-3}$. 
\end{lemma}

\begin{proof}
Let $\psi$ be an eigenfunction of $H$ with an eigenvalue $\lambda<0$. By virtue of \eqref{AC} and real interpolation, it suffices to show $\psi\in L^{q,2}$. For a given $\ep>0$, we decompose $V=v_1+v_2$ with $v_1\in C_0^\infty(\R^n)$ and $\norm{v_2}_{L^{n/2,\infty}}\le\ep$. We first let $\frac{2n}{n-1}< q<\frac{2n}{n-3}$. By Sobolev's inequality and Proposition \ref{proposition_free_Sobolev_1}, one has
\begin{align*}
\norm{R_0(\lambda)v_1\psi}_{L^{q}}&\lesssim \norm{R_0(\lambda)v_1\psi}_{\H^{n(\frac12-\frac1q)}}
\le C_\lambda \norm{v_1\psi}_{L^{2}}\le  C_\lambda \norm{v_1}_{L^{\infty}}\norm{\psi}_{L^2},\\
\norm{R_0(\lambda )v_2}_{\mathbb B(L^{q})}&\lesssim \norm{v_2}_{L^{\frac n2,\infty}}.
\end{align*}
For $\ep>0$ small enough, $I+R_0(\lambda)v_2$ thus is invertible on $L^{q}$ and 
$$
\psi=-R_0(\lambda)V\psi=R_0(\lambda)v_1\psi-R_0(\lambda)v_2\psi=-(I+R_0(\lambda)v_2)^{-1}R_0(\lambda)v_1\psi\in L^{q}. 
$$
Next, since $R_0(\lambda)\in \mathbb B(L^{p})$ for all $1<p<\infty$, we learn by H\"older's inequality that
$$
\norm{\psi}_{L^{p}}=\norm{R_0(\lambda)V\psi}_{L^{p}}\le C_\lambda\norm{V\psi}_{L^{p}}\le C_\lambda \norm{V}_{L^{\frac n2,\infty}}\norm{\psi}_{L^{q}} 
$$ 
if $\frac1p-\frac1q=\frac2n$. This shows $\psi\in L^p$ for all $\frac{2n}{n+3}<p<\frac{2n}{n+1}$. Interpolating these two cases, we conclude that $\psi\in L^q$ for all $\frac{2n}{n+3}< q< \frac{2n}{n-3}$. 
\end{proof}

\begin{proof}[Proof of Theorem \ref{theorem_KRS_2}]
Assume that $\mathcal E\cap[0,\infty)=\emptyset$. Then one can find  $\delta>0$ small enough such that $\dist(\mathcal E_\delta,[0,\infty))\ge \delta/2$. The proof is divided into two cases: (i) $z\in \C\setminus([0,\infty)\cup\mathcal E_{\delta})$ or (ii) $z\in \mathcal E_{\delta}$. For the case when $z\in \C\setminus([0,\infty)\cup\mathcal E_{\delta})$, 
since $\frac{2n}{n-1}<q,p'<\frac{2n}{n-3}$ and 
$
P_jR(z)=(\lambda_j-z)^{-1}\<\cdot,\psi_{j}\>\psi_{j}
$, 
Lemma \ref{lemma_KRS_4} implies
$$
\norm{P_jR(z)f}_{L^{p',2}}\le \delta^{-1}\norm{\psi_j}_{L^{q,2}}\norm{\psi_j}_{L^{p',2}}\norm{f}_{L^{p,2}}
$$
which, together with Theorem \ref{theorem_KRS_1} and the formula \eqref{AC}, gives us the desired bound
\begin{align}
\label{proof_theorem_KRS_2_1}
\sup_{z\in \C\setminus([0,\infty)\cup\mathcal E_{\delta})}\norm{P_{\mathrm{ac}} (H)R(z)}_{\mathbb B(L^{p,2},L^{q,2})}\lesssim\delta^{-1}.
\end{align}
When $z\in \mathcal E_{\delta}$, we use twice  the first resolvent equation
$
R(z)=R(z')-(z-z')R(z')R(z)
$
to write
$$
P_{\mathrm{ac}}(H) R(z)=P_{\mathrm{ac}}(H) R(M)+(z+M)P_{\mathrm{ac}}(H) R(M)^2+(z+M)^2R(M)P_{\mathrm{ac}}(H) R(z)R(M),
$$
where we have taken $M<\inf\sigma(H)-1$. Note that $|z+M|\le 2|M|+\delta$ for $z\in \mathcal E_\delta$ since $\mathcal E$ is a bounded set in $\R$. Moreover, we learn by Lemma \ref{lemma_KRS_4} and Corollary \ref{corollary_KRS_2} and Theorem \ref{theorem_interpolation_2} that
\begin{align*}
&\norm{P_{\mathrm{ac}}(H) R(M)}_{\mathbb B(L^{p,2},L^{q,2})}\le \norm{P_{\mathrm{ac}}(H)}_{\mathbb B(L^{q,2})}\norm{R(M)}_{\mathbb B(L^{p,2},L^{q,2})}\le C_M,\\
&\norm{R(M)}_{\mathbb B(L^2,L^{q,2})}+\norm{R(M)}_{\mathbb B(L^{p,2},L^2)}\le C_M
\end{align*}
with some $C_M$ being independent of $z$. It follows from these two bounds and the trivial $L^2$-bound
$$
\sup_{z\in \mathcal E_\delta}\norm{P_{\mathrm{ac}}(H) R(z)}_{\mathbb B(L^2)}\le \dist(\mathcal E_\delta,[0,\infty))^{-1}\le 2\delta^{-1}
$$
that there exists $C_{M,\delta}>0$, independent of $z$, such that
\begin{align}
\label{proof_theorem_KRS_2_2}
\sup_{z\in \mathcal E_\delta}\norm{P_{\mathrm{ac}}(H) R(z)}_{\mathbb B(L^{p,2},L^{q,2})}\le C_{M,\delta}.
\end{align}
The assertion of the theorem then follows from \eqref{proof_theorem_KRS_2_1} and \eqref{proof_theorem_KRS_2_2}. 
\end{proof}

\section{Kato smoothing and Strichartz estimates}
\label{section_Strichartz}
This section is devoted to the proof of Theorems \ref{theorem_smoothing_1} and \ref{theorem_Strichartz_1}. 
We first prepare several lemmas. Let $e^{it\Delta}$ be the free Schr\"odinger unitary group and define 
$$
\Gamma_0F(t):=\int_0^t e^{i(t-s)\Delta}F(s)ds,\quad F\in L^1_{\loc}(\R;L^2(\R^n)).
$$
The estimates for the free Schr\"odinger equation used in this section are summarized as follows: 

\begin{lemma}
\label{lemma_Strichartz_1}
Let $(p,q)$ satisfy \eqref{admissible}, $(p_s,q_s)$ be as in \eqref{p_q_s} and $\rho>1/2$. Then
\begin{align}
\label{lemma_Strichartz_1_1}
\norm{e^{it\Delta}\psi}_{L^p_tL^{q,2}_x}&\lesssim \norm{\psi}_{L^2_x},\\
\label{lemma_Strichartz_1_2}
\norm{\Gamma_0F}_{L^2_tL^{q_s,2}_x}&\lesssim \norm{F}_{L^2_tL^{p_s,2}_x}\quad \text{for}\ \frac{n}{2(n-1)}<s<\frac{3n-4}{2(n-1)},\\
\label{lemma_Strichartz_1_3}
\norm{\Gamma_0F}_{L^2_tL^{q_s}_x}&\lesssim \norm{F}_{L^2_tL^{p_s}_x}\quad  \text{for}\ s=\frac{n}{2(n-1)},\ \frac{3n-4}{2(n-1)},\\
\label{lemma_Strichartz_1_4}
\norm{\<x\>^{-\rho}|D|^{1/2}e^{it\Delta}\psi}_{L^2_tL^2_x}&\lesssim \norm{\psi}_{L^2_x},\\
\label{lemma_Strichartz_1_5}
\norm{\<x\>^{-\rho}|D|^{1/2}\Gamma_0F}_{L^2_tL^2_x}&\lesssim \norm{F}_{L^2_tL^{\frac{2n}{n+2},2}_x}.
\end{align}
\end{lemma}


\begin{proof}
\eqref{lemma_Strichartz_1_1} for $p>2$ is due to \cite{Str,GiVe}. \eqref{lemma_Strichartz_1_1} with $p=2$ and \eqref{lemma_Strichartz_1_2} with $s=1$ were settle by \cite{KeTa}. \eqref{lemma_Strichartz_1_2} was proved independently by \cite{Fos} and \cite{Vil}. \eqref{lemma_Strichartz_1_3} was settled recently by \cite{KoSe}. Kato-smoothing \eqref{lemma_Strichartz_1_4} was proved by \cite{KPV}. Finally, \eqref{lemma_Strichartz_1_5} can be found in \cite[Lemma 3.2]{Miz3}. 
\end{proof}

The following lemma, which was proved by Kato \cite{Kat} (see also \cite{ReSi,Dan}), shows the equivalence of uniform weighted resolvent estimate and Kato smoothing estimate.

\begin{lemma}
\label{lemma_Kato_smoothing_1}
Let $L$ be a self-adjoint operator on a Hilbert space $\H$, $A$ a densely defined closed operator on $\H$, $a>0$. Then the following two estimates are equivalent to each other: 
\begin{align*}
|\<\Im (L-z)^{-1} A^*u,A^*u\>_\H|&\le a\norm{u}_{\H}^2,\quad u\in D(A^*),\ z\in \C\setminus\R,\\
\norm{Ae^{-itL}v}_{L^2_t\H}&\le 2\sqrt a\norm{v}_{\H},\quad v\in \H. 
\end{align*}
\end{lemma}

The following concerns the equivalence of Sobolev norms generated by $\Delta$ and $H$. 

\begin{lemma}
\label{lemma_equivalence_2}
Assume that $\mathcal E\cap [0,\infty)=\emptyset$ and $0\le s<3/2$. Then
\begin{align}
\norm{(-\Delta+M)^{s/2}(H+M)^{-s/2}}_{\mathbb B(L^2)}+
\norm{(H+M)^{s/2}(-\Delta+M)^{-s/2}}_{\mathbb B(L^2)}<\infty.
\end{align}
\end{lemma}

\begin{proof}
The proof will be given in the next section. 
\end{proof}

Recall that $\<\cdot,\cdot\>_T$ is the inner product in $L^2_TL^2_x$ defined by
$
\<F,G\>_T=\int_{-T}^T\<F(t),G(t)\>dt
$. It is not hard to check that $\<\Gamma_HF,G\>_T=\<F,\Gamma_H^*G\>_T$ with 
$$
\Gamma_H^*G(t)=\mathds1_{[0,\infty)}(t)\int_t^Te^{-i(t-s)H}G(s)ds-\mathds1_{(-\infty,0]}(t)\int_{-T}^te^{-i(t-s)H}G(s)ds.
$$
The following lemma gives the  rigorous definition of Duhamel's formula (in the sense of forms). 

\begin{lemma}
\label{lemma_Duhamel_1}
Let $1/2<s<3/2$, $V_1\in L^{n/s,\infty}_0(\R^n;\R)$ and $V_2\in L^{n/(2-s),\infty}_0(\R^n;\R)$ be  such that $V=V_1V_2$. Then, for all $\psi \in L^2$ and all simple functions $F,G:\R\to \S$, 
\begin{align}
\label{lemma_Duhamel_1_1}
\<e^{-itH}P_{\mathrm{ac}}(H)\psi,G\>_T
&=\<e^{it\Delta}P_{\mathrm{ac}}(H)\psi,G\>_T-i\<V_1P_{\mathrm{ac}}(H)e^{-itH}\psi,V_2\Gamma_0^*G\>_T,\\
\label{lemma_Duhamel_1_2}
\<\Gamma_HP_{\mathrm{ac}}(H)F,G\>_T
&=\<\Gamma_0P_{\mathrm{ac}}(H)F,G\>_T-i\<V_1\Gamma_HP_{\mathrm{ac}}(H)F,V_2\Gamma_0^*G\>_T,\\
\label{lemma_Duhamel_1_3}
&=\<\Gamma_0F,P_{\mathrm{ac}}(H)G\>_T-i\<V_2\Gamma_0F,V_1\Gamma_H^*P_{\mathrm{ac}}(H)G\>_T.
\end{align}
\end{lemma}

\begin{proof}
The proof is basically same as that in \cite[Proposition 4.4]{BoMi} where the case $s=1$ was considered. We shall show \eqref{lemma_Duhamel_1_2}, otherwise the proof being similar. We start from the formula
\begin{align}
\label{proof_Duhamel_1_1}
\<e^{-itH}P_{\mathrm{ac}}(H)u,v\>-\<e^{it\Delta}P_{\mathrm{ac}}(H)u,v\>=-i\int_0^t\<V_1e^{-i\tau H}P_{\mathrm{ac}}(H)u,V_2e^{i(t-\tau)\Delta}v\>d\tau
\end{align}
for $u,v\in \S$, which follows by computing $\frac{d}{dt}\<e^{-itH}P_{\mathrm{ac}}(H)u,e^{it\Delta}v\>$. Here note that the HLS inequality \eqref{HLS} and Lemma \ref{lemma_equivalence_2} yield
\begin{align*}
&|\<V_1e^{-i\tau H}P_{\mathrm{ac}}(H)u,V_2e^{i(t-\tau )\Delta}v\>|\\
&\lesssim \norm{V_1}_{L^{\frac ns,\infty}}\norm{V_2}_{L^{\frac{n}{2-s},\infty}}\norm{(-\Delta+1)^{s/2}u}_{L^2}\norm{(-\Delta+1)^{1-s/2}v}_{L^2}<\infty
\end{align*}
and, hence,  the right hand side of \eqref{proof_Duhamel_1_1} makes sense. 
Changing $t$ by $t-s$, plugging $u=F(s)$, $v=G(t)$ and integrating in $s$ over $[0,t]$, we obtain
\begin{align*}
&\<\Gamma_HP_{\mathrm{ac}}(H)F(t),G(t)\>-\<\Gamma_0P_{\mathrm{ac}}(H)F(t),G(t)\>\\
&=-i\int_0^t\int_s^t\<V_1e^{-i(\tau-s)H}P_{\mathrm{ac}}(H)F(s),V_2e^{i(\tau-t)\Delta}G(t)\>d\tau dt,
\end{align*}
where, by the same argument as above, the integrand of the right hand side is finite and thus integrable in $(\tau,s)\in [0,t]^2$. Therefore, by Fubini's theorem, 
\begin{equation}
\begin{aligned}
\label{proof_Duhamel_1_2}
&\<\Gamma_HP_{\mathrm{ac}}(H)F(t),G(t)\>-\<\Gamma_0P_{\mathrm{ac}}(H)F(t),G(t)\>\\
&=-i\int_0^t\<V_1\Gamma_HP_{\mathrm{ac}}(H)F(\tau),V_2e^{i(\tau-t)\Delta}G(t)\>d\tau.
\end{aligned}
\end{equation}
Finally, observing from the same argument as above that $|\<V_1\Gamma_HF(\tau),V_2e^{i(\tau-t)\Delta}G(t)\>|$ is finite, we integrate \eqref{proof_Duhamel_1_2} in $t$ and use Fubini's theorem to obtain the desired formula \eqref{lemma_Duhamel_1_2}. 
\end{proof}

\begin{remark}
When $s=1$, the identities \eqref{lemma_Duhamel_1_1}, \eqref{lemma_Duhamel_1_2} and \eqref{lemma_Duhamel_1_3} also hold for all $F,G\in L^1_{\loc}L^2$ (see \cite[Proposition 4.4]{BoMi}). 
\end{remark}

Using these lemmas, we first prove Kato smoothing estimates. 

\begin{proof}[Proof of Theorem \ref{theorem_smoothing_1}]
The following argument is basically same as that in \cite{BPST2}. With the above remark at hand, we use \eqref{lemma_Duhamel_1_1} with $G$ replaced by $|D|^{1/2}\<x\>^{-\rho}G$ to obtain
\begin{align*}
&\<\<x\>^{-\rho}|D|^{1/2}e^{-itH}P_{\mathrm{ac}}(H)\psi,G\>_T\\
&= \<\<x\>^{-\rho}|D|^{1/2}e^{it\Delta}P_{\mathrm{ac}}(H)\psi,G\>_T-i\<V_1P_{\mathrm{ac}}(H)e^{-itH}\psi,V_2\Gamma_0^*|D|^{1/2}\<x\>^{-\rho}G\>_T
\end{align*}
for all $\psi \in L^2$ and a simple function $G(t):\R\to \S$ . By \eqref{lemma_Strichartz_1_4}, the first term obeys
\begin{align}
\label{proof_theorem_smoothing_1_1}
|\<\<x\>^{-\rho}|D|^{1/2}e^{it\Delta}P_{\mathrm{ac}}(H)\psi,G\>_T|\lesssim\norm{\psi}_{L^2}\norm{G}_{L^2_tL^2_x}
\end{align}
uniformly in $T>0$. On the other hand, we learn by the dual estimate of \eqref{lemma_Strichartz_1_5} that
\begin{align}
\label{proof_theorem_smoothing_1_2}
|\<V_1P_{\mathrm{ac}}(H)e^{-itH}\psi,V_2\Gamma_0^*G\>_T|\lesssim \norm{V_1P_{\mathrm{ac}}(H)e^{-itH}\psi}_{L^2_tL^2_x}\norm{G}_{L^2_tL^2_x}
\end{align}
uniformly in $T>0$. For the term $\norm{V_1P_{\mathrm{ac}}(H)e^{-itH}\psi}_{L^2_tL^2_x}$, we use Lemma \ref{lemma_Kato_smoothing_1} to deduce
\begin{align}
\label{proof_theorem_smoothing_1_3}
\norm{V_1P_{\mathrm{ac}}(H)e^{-itH}\psi}_{L^2_tL^2_x}\lesssim \norm{\psi}_{L^2_x}
\end{align}
from the following uniform weighted resolvent estimate
$$
\sup_{z\in \C\setminus\R}\norm{V_1P_{\mathrm{ac}}(H)R(z)P_{\mathrm{ac}}(H)V_1}_{\mathbb B(L^2)}<\infty
$$
which is a consequence of Theorem \ref{theorem_KRS_2} and H\"older's inequality \eqref{Holder}, where note that $P_{\mathrm{ac}}(H)^2=P_{\mathrm{ac}}(H)$ since $P_{\mathrm{ac}}(H)$ is an orthogonal projection. Finally, \eqref{proof_theorem_smoothing_1_1}--\eqref{proof_theorem_smoothing_1_3} imply
$$
|\<\<x\>^{-\rho}|D|^{1/2}e^{-itH}P_{\mathrm{ac}}(H)\psi,G\>_T|\lesssim \norm{\psi}_{L^2}\norm{G}_{L^2_tL^2_x}
$$
which, together with duality and density argument, gives us the assertion.
\end{proof}

In order to prove Strichartz estimates, we need one more lemma. 

\begin{lemma}
\label{lemma_Strichartz_3}
Assume $\mathcal E\cap [0,\infty)=\emptyset$. Then, for any $1/2< s< 3/2$ there exists $C>0$ such that, for all $w\in L^{n/(2-s),\infty}$, $\chi\in C_0^\infty(\R^n)$ and $T>0$, 
\begin{align}
\label{lemma_Strichartz_3_1}
\norm{\chi\Gamma_H P_{\mathrm{ac}}(H) wF}_{L^2_TL^2_x}\le C \norm{\chi}_{L^{\frac ns,\infty}}\norm{w}_{L^{\frac{n}{2-s},\infty}}\norm{F}_{L^2_TL^2_x}.
\end{align}
\end{lemma}

\begin{proof}
The proof is essentially based on the argument by D'Ancona \cite[Theorem 2.3]{Dan}. At first note that it suffices to show \eqref{lemma_Strichartz_3_1} with $[-T,T]$ replaced by $\R$. Indeed, since $s\in [-T,T]$ if $t\in [-T,T]$ and $s\in [0,t]$ (or $s\in [t,0]$), \eqref{lemma_Strichartz_3_1} with $[-T,T]$ replaced by $\R$ implies
\begin{align*}
\norm{\chi\Gamma_H P_{\mathrm{ac}}(H) wF}_{L^2_TL^2_x}
\lesssim\norm{\mathds1_{[-T,T]}(s)F}_{L^2_tL^2_x}=\norm{F}_{L^2_TL^2_x}. 
\end{align*}
We may assume, by a density argument, that $F(t):\R\to \S$ is a simple function. Set $A_1=\chi(x)P_{\mathrm{ac}}(H)$ and $A_2=wP_{\mathrm{ac}}(H)$. 
For a function $v(t):\R\to L^2$, $\wtilde v$ denotes its Laplace transform:
$$
\wtilde v(z)=\pm\int_0^{\pm\infty}e^{izt}v(t)dz,\quad \pm\Im z>0. 
$$
A direct calculation yields that if $v(t)=\Gamma_HA_2^*F(t)$ then
$
\wtilde v(z)=-iR(z)A_2^*\wtilde F(z)
$, 
where the identity $\wtilde{A_2^*F}=A_2^*\wtilde F$ follows from the estimate $\norm{A_2F}_{L^1_{\loc}L^2_x}\lesssim \norm{w}_{L^{\frac{n}{2-s},\infty}}\norm{F}_{L^1_{\loc}\H^{2-s}}<\infty$ and Hille's theorem \cite[Theorem 3.7.12]{HiPh}.  Also we see that $v(t)\in D(A_1)$ for each $t$. Indeed, writing $F(t)=\sum_{j=1}^N\mathds1_{E_j}(t)f_j$ with some $f_j\in \S(\R^n)$, we have for each $t$
$$
\norm{A_1v(t)}_{L^2}\le \sum_{j=1}^N\int_0^{|t|} \norm{A_1e^{isH}e^{-itH}P_{\mathrm{ac}}(H)wf_j}_{L^2}ds\lesssim |t|\norm{w}_{L^{\frac {n}{2-s},\infty}}\sum_{j=1}^N\norm{f_j}_{\H^{2-s}}<\infty. 
$$
Then one can use Parseval's theorem to obtain
$$
\pm\int_0^{\pm\infty}e^{-2\ep|t|} \<v(t),A_1^*G(t)\>dt=2\pi\int_\R\<\wtilde v(\lambda\pm i\ep),A_1^*\wtilde{G}(\lambda\pm i\ep)\>d\lambda,\quad \ep>0,
$$
for any simple function $G:\R\to \S$. By virtue of uniform Sobolev estimates \eqref{theorem_KRS_2_2} with $(p,q)=(\frac{2n}{n+2(2-s)},\frac{2n}{n-2s})$ and H\"older's inequality \eqref{Holder}, the integrand of  the right hand  side obeys
$$
|\<\wtilde v(\lambda\pm i\ep),A_1^*\wtilde{G}(\lambda\pm i\ep)\>|\le \norm{\chi}_{L^{\frac n2,\infty}}\norm{w}_{L^{\frac{n}{2-s},\infty}}\norm{\wtilde F(\lambda\pm i\ep)}_{L^2_x}\norm{\wtilde G(\lambda\pm i\ep)}_{L^2_x}. 
$$
Applying again Parseval's theorem, we have
\begin{align*}
&\left|\int_0^{\pm\infty}e^{-2\ep|t|} \<v(t),A_1^*G(t)\>dt\right|=\left|\int_\R\<\wtilde v(\lambda\pm i\ep),A_1^*\wtilde{G}(\lambda\pm i\ep)\>d\lambda
\right|\\
&\lesssim \norm{\chi}_{L^{\frac n2,\infty}}\norm{w}_{L^{\frac{n}{2-s},\infty}}\norm{\wtilde F(\lambda\pm i\ep)}_{L^2_\lambda L^2_x}\norm{\wtilde G(\lambda\pm i\ep)}_{L^2_\lambda L^2_x}\\
&\lesssim\norm{\chi}_{L^{\frac n2,\infty}}\norm{w}_{L^{\frac{n}{2-s},\infty}}\norm{e^{-\ep|t|}F(t)}_{L^2(\R_{\pm};L^2(\R^n))}\norm{e^{-\ep|t|}G(t)}_{L^2(\R_{\pm};L^2(\R^n))},
\end{align*}
which, together with the density of simple functions with values in $\S$, shows
$$
\norm{e^{-\ep|t|}A_1\Gamma_HA_2F}_{L^2_tL^2_x}\lesssim \norm{\chi}_{L^{\frac n2,\infty}}\norm{w}_{L^{\frac{n}{2-s},\infty}}\norm{e^{-\ep|t|}F}_{L^2_tL^2_x},\quad F\in L^2_tL^2_x.
$$
The result then follows by letting $\ep\to0$. 
\end{proof}

\begin{remark}
\label{remark_Strichartz_4}
If $1/2<s\le 1$, \eqref{lemma_Strichartz_3_1} also holds for any $\chi\in L^{n/s,\infty}$. The proof is completely same. When $1<s<3/2$, we do not, a priori, know $\chi e^{-itH}P_{\mathrm{ac}}(H)wF(s)\in L^2_x$ for each $t,s$ under the condition $\chi\in L^{n/s,\infty}$ only, even if $F:\R\to \S$. This is the reason why we have assumed $\chi\in C_0^\infty$. We however stress that Lemma \ref{lemma_Strichartz_3} is sufficient for our purpose. 
\end{remark}

We are now ready to show our Strichartz estimates.

\begin{proof}[Proof of Theorem \ref{theorem_Strichartz_1}]
Using \eqref{lemma_Strichartz_1_1} and \eqref{lemma_Strichartz_1_2} with $s=1$ instead of \eqref{lemma_Strichartz_1_4} and \eqref{lemma_Strichartz_1_5}, respectively, one can see that the proof of the homogeneous endpoint Strichartz estimate of the form
\begin{align}
\label{proof_theorem_Strichartz_1_0}
\norm{e^{-itH}P_{\mathrm{ac}}(H)\psi}_{L^{2}_tL^{\frac{2n}{n-2},2}_x}\lesssim \norm{\psi}_{L^2}
\end{align}
is similar to that of Theorem \ref{theorem_smoothing_1} and even easier than that of \eqref{theorem_Strichartz_1_2}. We thus omit the proof. 

We shall show \eqref{theorem_Strichartz_1_2}. Let $\frac{n}{2(n-1)} <s< \frac{3n-4}{2(n-1)}$, $V_1\in L^{n/s,\infty}_0$ and $V_2\in L^{n/(2-s),\infty}_0$ be real-valued such that $V=V_1V_2$. 
Take a sequence $V_{1,j}\in C_0^\infty$ such that $\norm{V_{1}-V_{1,j}}_{L^{n/s,\infty}}\to0$. Let $F:\R\to \S$ be a simple function in $t$. As in the proof of Lemma \ref{lemma_Duhamel_1}, we see that $\Gamma_HP_{\mathrm{ac}}(H)F\in L^2_TL^{q_s,2}_x$ for each $T>0$ by Lemma \ref{lemma_equivalence_2}. Then, by the duality argument, we have
\begin{align}
\label{proof_theorem_Strichartz_1_1}
\norm{\Gamma_HP_{\mathrm{ac}}(H)F}_{L^2_TL^{q_s,2}_x}\lesssim \sup\{|\<\Gamma_H P_{\mathrm{ac}}(H)F,G\>_T|\ |\ \norm{G}_{L^2_TL^{q_s',2}_x}=1\}
\end{align}
where we may assume by density argument that $G:\R\to \S$ is a simple function. Then, it follows from Duhamel's formula \eqref{lemma_Duhamel_1_2}, \eqref{lemma_Strichartz_1_2}, Lemma \ref{lemma_KRS_4} and H\"older's inequality \eqref{Holder} that
\begin{align*}
|\<\Gamma_H P_{\mathrm{ac}}(H)F,G\>_T|
&\lesssim \norm{P_{\mathrm{ac}}(H)}_{\mathbb B(L^{p_s,2})}\norm{F}_{L^2_tL^{p_s,2}_x}+\norm{V_{1,j}\Gamma_HP_{\mathrm{ac}}(H)F}_{L^2_TL^2_x}\norm{V_2}_{L^{\frac{n}{2-s},\infty}}\\
&\quad\quad+\norm{V_1-V_{1,j}}_{L^{\frac ns,\infty}}\norm{\Gamma_HP_{\mathrm{ac}}(H)F}_{L^2_TL^{q_s,2}_x}
\end{align*}
uniformly in $T>0$. Taking $j$ large enough (which can be taken independently of $T$), the last term can be absorbed into the left hand side of \eqref{proof_theorem_Strichartz_1_1}, implying
$$
\norm{\Gamma_HP_{\mathrm{ac}}(H)F}_{L^2_TL^{q_s,2}_x}\lesssim \norm{F}_{L^2_tL^{p_s,2}_x}+\norm{V_{1,j}\Gamma_HP_{\mathrm{ac}}(H)F}_{L^2_TL^2_x}
$$
uniformly in $T>0$. To deal with the term $\norm{V_{1,j}\Gamma_HP_{\mathrm{ac}}(H)F}_{L^2_TL^2_x}$, we use  \eqref{lemma_Duhamel_1_3} to write
$$
\<V_{1,j}\Gamma_HP_{\mathrm{ac}}(H)F,\wtilde G\>_T
=\<\Gamma_0F,P_{\mathrm{ac}}(H)V_{1,j}\wtilde G\>_T
-i\<V_2\Gamma_0F,V_1\Gamma_H^*P_{\mathrm{ac}}(H)V_{1,j}\wtilde G\>_T
$$
for all simple function $\widetilde G:\R\to \S$ satisfying $\norm{\wtilde G}_{L^2_TL^2_x=1}=1$. 
By \eqref{lemma_Strichartz_1_2} the first term enjoys
$$
|\<\Gamma_0F,P_{\mathrm{ac}}(H)V_{1,j}\wtilde G\>_T|\lesssim \norm{V_{1,j}}_{L^{\frac ns,\infty}}\norm{F}_{L^2_tL^{p_s,2}_x}\lesssim \norm{F}_{L^2_tL^{p_s,2}_x}
$$
uniformly in $T>0$ and $j$. On the other hand, since $V_2\Gamma_H^*P_{\mathrm{ac}}(H)V_{1,j}\wtilde G\in L^2_TL^2_x$ by Lemma \ref{lemma_Strichartz_3} and $V_1\Gamma_0F\in L^2_TL^2_x$ by \eqref{lemma_Strichartz_1_2}, the last term can be rewritten in the form
$$
\<V_2\Gamma_0F,V_1\Gamma_H^*P_{\mathrm{ac}}(H)V_{1,j}\wtilde G\>_T
=\<V_1\Gamma_0F,V_2\Gamma_H^*P_{\mathrm{ac}}(H)V_{1,j}\wtilde G\>_T.
$$
Using \eqref{lemma_Strichartz_1_2}, Lemma \ref{lemma_Strichartz_3} and a duality argument, we then obtain
$$
|\<V_1\Gamma_0F,V_2\Gamma_H^*P_{\mathrm{ac}}(H)V_{1,j}\wtilde G\>_T|\lesssim \norm{F}_{L^2_tL^{p_s,2}_x}.
$$
Putting it all together, we conclude that
$$
\norm{\Gamma_HP_{\mathrm{ac}}(H)F}_{L^2_TL^{q_s,2}_x}\lesssim \norm{F}_{L^{2}_tL^{p_s,2}_x}
$$
uniformly in $T>0$, which implies the desired estimates \eqref{theorem_Strichartz_1_2} for $\frac{n}{2(n-1)} <s< \frac{3n-4}{2(n-1)}$. The cases $s=\frac{n}{2(n-1)}$ or $\frac{3n-4}{2(n-1)}$ can be obtained analogously by using \eqref{lemma_Strichartz_1_3} instead of \eqref{lemma_Strichartz_1_2}. 
\end{proof}

\section{Spectral multiplier theorem}
\label{section_multiplier}

This section is devoted to the proof of Lemma \ref{lemma_equivalence_2} and Theorem \ref{theorem_multiplier_1}. Proofs are based on an abstract method by Chen et al \cite{COSY} which, in the Euclidean case, can be stated as follows. 

\begin{proposition}[{\cite[Theorem A]{COSY}}]
\label{proposition_COSY}
Let $1\le p_0<2$ and $1\le q\le \infty$. Let $L$ be a non-negative self-adjoint operator on $L^2(\R^n)$ satisfying following two conditions:
\begin{itemize}
\item Davies-Gaffney's estimate: for any open sets $U_j\subset \R^n$ and $\psi_j\in L^2(U_j)$, $j=1,2$
\begin{align}
\label{DG}
|\<e^{-tL}\psi_1,\psi_2\>|\le \exp\Big(-\frac{d(U_1,U_2)^2}{4t}\Big)\norm{\psi_1}_{L^2}\norm{\psi_2}_{L^2},
\end{align}
where $d(U_1,U_2):=\inf_{x_1\in U_1,x_2\in U_2}|x_1-x_2|$ is the distance between $U_1$ and $U_2$. 
\item Stein-Tomas type restriction estimate: for any $a>0$ and any bounded Borel function $F_0$ on $\R$ supported in $[0,a]$, $F_0(\sqrt{L})\in \mathbb B(L^{p_0},L^2)$ and 
\begin{align}
\label{ST}
\norm{F_0(\sqrt{L})\mathds1_{B(x,r)}}_{\mathbb B(L^{p_0},L^2)}\lesssim a^{n(1/p_0-1/2)}\norm{F_0(a\cdot)}_{L^q}
\end{align}
for all $x\in \R^n$ and $r\ge a^{-1}$, where $B(x,r)=\{y\ |\ |y-x|<r\}$. 
\end{itemize}
Then, for any bounded Borel function $F$ on $\R$ satisfying
\begin{align}
\label{Hormander}
|F|_{\mathcal W(\beta,q)}:=\sup_{t>0}\norm{\psi(\cdot)F(t\cdot)}_{\mathcal W^{\beta,q}(\R)}<\infty
\end{align}
with some nontrivial $\psi\in C_0^\infty$ supported in $(0,\infty)$ and $\beta>\max\{n(1/p_0-1/2),1/q\}$ such that $\beta$ is an integer if $q=\infty$, $F(\sqrt{L})$ is bounded on $L^p$ for all $p_0<p<p_0'$ and satisfies
$$
\norm{F(\sqrt{L})}_{\mathbb B(L^p)}\le C_\beta(|F|_{\mathcal W(\beta,q)}+|F(0)|). 
$$
\end{proposition}

Strictly speaking, instead of Davies-Gaffney's estimate, it was assumed in \cite{COSY} that $L$ satisfies the so-called finite-speed propagation property (see (\text{FS}) in pages 229 of \cite{COSY}). However, these two conditions are known to be equivalent (see \cite[Theorem 3.4]{CoSi}). Moreover, \eqref{DG} is always satisfied for non-negative Schr\"odinger operators $-\Delta+V(x)$ as shown by Coulhon-Sikora \cite{CoSi}. 

\begin{lemma}[{\cite[Theorem 3.3]{CoSi}}]
\label{lemma_CS}
Let $L=-\Delta+V(x)$ with real-valued $V\in L^1_{\loc}(\R^n)$ such that $L\ge0$ as a quadratic form. Then \eqref{DG} is satisfied. 
\end{lemma}

When $q=\infty$, \eqref{ST} can be replaced by a $L^p$-$L^2$ estimate of the Schr\"odinger semigroup. 

\begin{lemma}
\label{lemma_heat}
Let $1\le p_0<2$. Then \eqref{ST} with $q=\infty$ follows from 
\begin{align}
\label{lemma_heat_1}
\norm{e^{-t^2L}}_{\mathbb B(L^{p_0},L^2)}\lesssim t^{-n(\frac{1}{p_0}-\frac12)},\quad t>0. 
\end{align}
\end{lemma}

\begin{proof}
By \cite[Proposition 1.3]{COSY}, \eqref{ST} with $q=\infty$ is equivalent to 
$$
\norm{e^{-t^2L}\mathds1_{B(x,r)}}_{\mathbb B(L^{p_0},L^2)}\lesssim |B(x,r)|^{\frac{1}{p_0}-\frac12}(rt^{-1})^{n(\frac{1}{p_0}-\frac12)},\quad t>0,\ x\in \R^n,\ r\ge t,
$$
which clearly follows from \eqref{lemma_heat_1} since $|B(x,r)|\le C_nr^n$. 
\end{proof}

Now we show Lemma \ref{lemma_equivalence_2} whose proof is classical and based on Stein's complex interpolation theorem. Let us fix $M>|\inf\sigma(H)|+1$ so that $H+M\ge I$. A key observation is the following. 

\begin{lemma}
\label{lemma_proof_equivalence_2_1}
For any $\alpha\in \R$ and $\frac{2n}{n+3}<p<\frac{2n}{n-3}$, 
$
\norm{(H+M)^{i\alpha}}_{\mathbb B(L^p)}\le C_M\<\alpha\>^n. 
$ 
\end{lemma}

\begin{proof}
It is easy to see that $F(x)=x^{2i\alpha}$ satisfies $|F|_{\mathcal W({n,\infty})}\le C_n\<\alpha\>^n$ and $|F(0)|=1$. Let us fix $2n/(n+3)<p_0\le 2n/(n+2)$ arbitrarily. By virtue of Proposition \ref{proposition_COSY} and Lemmas \ref{lemma_CS} and \ref{lemma_heat}, it suffices to show that $L:=H+M$ satisfies \eqref{lemma_heat_1}. Decompose $e^{-t^2L}$ into the absolutely continuous part $e^{-t^2L}P_{\mathrm{ac}}(H)$ and the discrete part $\sum_{j=1}^Ne^{-t^2L}P_j$. 

For the discrete part, since $\lambda_j+M\ge1$, we learn by Lemma \ref{lemma_KRS_4} that
\begin{align*}
\norm{e^{-t^2L}P_jf}_{L^2}=\norm{e^{-t^2(\lambda_j+M)}P_jf}_{L^2}\le e^{-t^2}\norm{\varphi_j}_{L^2}\norm{\varphi_j}_{L^{p_0'}}\norm{f}_{L^{p_0}}\lesssim e^{-t^2}\norm{f}_{L^{p_0}}. 
\end{align*}
On the other hand, it follows from the spectral decomposition theorem that
$$
e^{-t^2L}P_{\mathrm{ac}}(H)(e^{-t^2L}P_{\mathrm{ac}}(H))^*=e^{-2t^2L}P_{\mathrm{ac}}(H)=\int_0^\infty e^{-2t^2(\lambda+M)}dE_{H}(\lambda).
$$
Theorem \ref{theorem_spectral_measure_1} then implies
\begin{align*}
\norm{e^{-2t^2L}P_{\mathrm{ac}}(H)}_{\mathbb B(L^{p_0},L^{p'_0})}
\lesssim \int_0^\infty e^{-2t^2(\lambda+M)}\lambda^{\frac n2(\frac{1}{p_0}-\frac{1}{p'_0})-1}d\lambda
\lesssim t^{-n(\frac{1}{p_0}-\frac{1}{p'_0})}=t^{-2n(\frac{1}{p_0}-\frac12)}. 
\end{align*}
Since $\norm{e^{-t^2L}P_{\mathrm{ac}}(H)}_{\mathbb B(L^{p_0},L^2)}\le \norm{e^{-2t^2L}P_{\mathrm{ac}}(H)}_{\mathbb B(L^{p_0},L^{p_0'})}^{1/2}$ by the duality, \eqref{lemma_heat_1} follows. 
\end{proof}

\begin{proof}[Proof of Lemma \ref{lemma_equivalence_2}]
We may assume $1<s<3/2$ without loss of generality since the case when $0\le s\le1$ follows from Stein's complex interpolation \cite{Ste} and the estimate $$
\norm{(-\Delta+M)^{1/2}(H+M)^{-1/2}}_{\mathbb B(L^2)}+\norm{(-\Delta+M)^{-1/2}(H+M)^{1/2}}_{\mathbb B(L^2)}<\infty
$$
which is a consequence of the fact that the form domain of $H$ is $\H^1$. 

For $f,g\in \S$, we consider a function 
$
G(z)=\<(H+M)^{-z}f,(-\Delta+M)^{z}g\>
$
which is continuous on $0\le \Re z\le1$ and analytic in $0< \Re z<1$. By Corollary \ref{corollary_KRS_2}  and Lemma \ref{lemma_proof_equivalence_2_1}, we have for $\frac{2n}{n+3}<r_1<\frac{2n}{n-3}$ and $\frac{2n}{n+3}<r_2<\frac{2n}{n+1}$,
\begin{align*}
|G(it)|
&\le \norm{(H+M)^{-it}f}_{L^{r_1}}\norm{(-\Delta+M)^{it}g}_{L^{r_1'}}
\lesssim \<t\>^{2n}\norm{f}_{L^{r_1}}\norm{g}_{L^{r_1'}},\\
|G(1+it)|
&\le \norm{(-\Delta+M)(H+M)^{-1-it}f}_{L^{r_2}}\norm{(-\Delta+M)^{it}g}_{L^{r_2'}}
\lesssim \<t\>^{2n}\norm{f}_{L^{r_2}}\norm{g}_{L^{r_2'}}, 
\end{align*}
where, since $(-\Delta+M)(H+M)^{-1}=1-V(H+M)^{-1}$, the second estimate can be verified as\begin{align*}
\norm{(-\Delta+M)(H+M)^{-1}}_{\mathbb B(L^{r_2})}
\le 1+\norm{V(H+M)^{-1}}_{\mathbb B(L^{r_2})}
\le 1+C_M\norm{V}_{L^{\frac{n}{2},\infty}}.
\end{align*}
Let $r_1=\frac{2n}{n-2s}$ and $r_2=\frac{2n}{n+2(2-s)}$. Since
$
1/2=(1-s/2)(1/r_1)+(s/2)(1/r_2)
$, we apply Stein's complex interpolation theorem to $G$, implying $|G(s/2)|\le C_t\norm{f}_{L^2}\norm{g}_{L^2}$. This gives us  $$\norm{(-\Delta+M)^{s/2}(H+M)^{-s/2}}_{\mathbb B(L^2)}<\infty.$$ Applying the same argument to a function $G(z)=\<(-\Delta+M)^{-z}f,(H+M)^{z}g\>$, we also have $\norm{(H+M)^{s/2}(-\Delta+M)^{-s/2}}_{\mathbb B(L^2)}<\infty$. 
\end{proof}

Next we shall show Theorem \ref{theorem_multiplier_1}. 

\begin{proof}[Proof of Theorem \ref{theorem_multiplier_1}]
Since $H$ is assumed to be non-negative, the Davies-Gaffney estimate \eqref{DG} is satisfied. It thus remains to check the Stein-Tomas type restriction estimate \eqref{ST} with $q=2$. Let $\frac{2n}{n+3}<p_0<\frac{2n}{n+2}$ and $F_0\in L^\infty(\R)$ be such that $\supp F_0\subset [0,a]$. By Theorem \ref{theorem_spectral_measure_1}, 
\begin{align*}
\norm{F_0(\sqrt{H})^2}_{\mathbb B(L^{p_0},L^{p_0'})}
&\lesssim \int_0^{a^2} |F_0(\sqrt{\lambda})|^2\lambda^{\frac n2(\frac{1}{p_0}-\frac{1}{p'_0})-1}d\lambda\\
&\lesssim \norm{F_0}_{L^2([0,a])}^2a^{n(\frac{1}{p_0}-\frac{1}{p'_0})-1}
\lesssim a^{n(\frac{1}{p_0}-\frac{1}{p'_0})}\norm{F_0(a\cdot)}_{L^2}^2.
\end{align*}
Finally, by the duality, we have $\norm{F_0(\sqrt{H})}_{\mathbb B(L^{p_0},L^2)}\le 
\norm{F_0(\sqrt{H})^2}_{\mathbb B(L^{p_0},L^{p_0'})}^{1/2}$ which, combined with the above estimate for $\norm{F_0(\sqrt{H})^2}_{\mathbb B(L^{p_0},L^{p_0'})}$, implies \eqref{ST} with $q=2$. 
\end{proof}

We conclude this section with two immediate consequences of Theorem \ref{theorem_multiplier_1}. 

\begin{corollary}
\label{corollary_multiplier_1}
Suppose that $\mathcal E\cap[0,\infty)=\emptyset$, $H\ge0$ and $0\le s<3/2$. Then
$$
\norm{(-\Delta)^{s/2}H^{-s/2}}_{\mathbb B(L^2)}+
\norm{H^{s/2}(-\Delta)^{-s/2}}_{\mathbb B(L^2)}<\infty.
$$
\end{corollary}

\begin{proof}
The proof is analogous to that of Lemma \ref{lemma_equivalence_1}. 
\end{proof}

\begin{corollary}
\label{corollary_multiplier_2}
Suppose that $\mathcal E\cap[0,\infty)=\emptyset$ and $H\ge0$. Let $\varphi\in C_0^\infty(\R)$ be such that $\supp \varphi\subset(1/2,2)$, $0\le \varphi\le1$ and 
$
\sum_{j\in \Z}\varphi(2^{-j}\lambda)=1
$ for all $\lambda>0$. Then, for any $2n/(n+3)<p<2n/(n-3)$, there exists $C_p>0$ such that
$$
C_p^{-1}\norm{f}_{L^p}\le \bignorm{\Big(\sum_{j\in \Z}|\varphi(2^{-j}H)f(x)|^2\Big)^{1/2}}_{L^p}\le C_p\norm{f}_{L^p}. 
$$
In particular, if $2\le p<2n/(n-3)$, then
$$
\norm{f}_{L^p}\lesssim \Big(\sum_{j\in \Z}\norm{\varphi(2^{-j}H)f}_{L^p}^2\Big)^{1/2}. 
$$
\end{corollary}

\begin{proof}
With Theorem \ref{theorem_multiplier_1} as hand, the corollary follows from a standard method by  Stein \cite{Ste2}. The proof is completely same as that for the usual Littlewood-Paley estimate and we omit it. 
\end{proof}

\section{Eigenvalue bounds}
\label{section_EB}

This section is devoted to the proof of Theorem \ref{theorem_EB_1}. The proof is based on a  method by Frank \cite{Fra1} and \cite{Fra2}. 
Recall that $W\in L^{n/2+\gamma}(\R^n;\C)$ with $0<\gamma<\infty$. Then $W$ is $H$-form compact. Indeed, taking $M>-\inf\sigma(H)$, we see that $|W|^{1/2}(1-\Delta)^{-1/2}$ is compact and $(1-\Delta)^{1/2}(H+M)^{-1/2}$ is bounded. Hence $|W|^{1/2}(H+M)^{-1/2}=|W|^{1/2}(1-\Delta)^{-1/2}(1-\Delta)^{1/2}(H+M)^{-1/2}$ is also compact. Then there exists a unique $m$-sectorial operator $H_W$ such that $D(H_W)\subset Q(H_W)=\H^1$ and $\<H_Wu,v\>=\<(H+W)u,v\>$ for $u\in D(H_W)$ and $v\in \H^1$; $D(H_W)$ is dense in $\H^1$; $\sigma(H_W)$ is contained in a sector $\{z\in \C\ |\ |\arg(z-z_0)|\le \theta\}$ with some $z_0\in \R$ and $\theta\in[0,\pi/2)$ (see \cite[Theorems VI.3.9 and VI.2.1]{Kat}). We fix a factorization $W=W_1W_2$ with $W_1=|W|^{1/2}\sgn W$ and $W_2=|W|^{1/2}$, where $\sgn W(x)=W(x)/|W(x)|$ if $W(x)\neq0$ and $\sgn W(x)=0$ if $W(x)=0$. Let $d(z)=\dist(z,[\infty)$. We begin with the following lemma. 

\begin{lemma}\label{lemma_EB_1}
Suppose that $E\in \C\setminus\sigma(H)$ is an eigenvalue of $H_W$. Then $-1$ is an eigenvalue of $W_1R(E)W_2$. Moreover, if $0< \gamma\le 1/2$, the same statement also holds for $E\in(0,\infty)\setminus\mathcal E$ with $R(E)$ replaced by $R(E+ i0)$. 
\end{lemma}

\begin{proof}
We show the lemma for the case $E\in(0,\infty)\setminus\mathcal E$ only, since, in  the case $E\in \C\setminus \sigma(H)$, the lemma is a consequence of the well-known Birman-Schwinger principle (see, {\it e.g.}, \cite[Section 4]{Fra2}) and the proof is easier. Let $f\in \Ker_{L^2}(H_W-E)$. We let $\varphi\in \S$ and plug $v=R(E-i\ep)W_1\varphi\in \H^1$ into the identity $\<(H-E)f,v\>+\<W_1f,W_2v\>=0$, letting $\ep\searrow0$ and then using Corollary \ref{corollary_KRS_1} (2) to obtain 
$$
\<W_1f,\varphi\>+\<W_1R(E+i0)W_2W_1 f,\varphi\>=0.
$$
Since $\norm{W_1f}_{L^2}\lesssim \norm{W_1}_{L^{n+2\gamma}}\norm{f}_{\H^1}<\infty$, this shows $W_1f\in \Ker_{L^2}(I+W_1R(E+i0)W_2)$. 
\end{proof}

Since $W_1R(E)W_2$ is a compact operator on $L^2$, if $-1$ is an eigenvalue of $W_1R(E)W_2$ then $\norm{W_1R(E)W_2}_{\mathbb B(L^2)}\ge 1$ at least. 
With this remark at hand, it is easy to see that Theorem \ref{theorem_EB_1} follows from the following lemma.

\begin{lemma}
\label{lemma_EB_2}
For any $\delta>0$ and $0\le \gamma\le 1/2$, one has 
\begin{align}
\label{lemma_EB_2_1}
\norm{W_1R(z)W_2}_{\mathbb B(L^2)}\le C_\delta |z|^{-\frac{\gamma}{n/2+\gamma}}\norm{W}_{L^{n/2+\gamma}},\quad z\in \C\setminus \mathcal E_\delta,
\end{align}
where $R(z)$ is replaced by $R(z+i0)$ if $z\in (0,\infty)\setminus \mathcal E_\delta$. Moreover, for any $\gamma>1/2$, 
\begin{align}
\label{lemma_EB_2_2}
\norm{W_1R(z)W_2}_{\mathbb B(L^2)}\le C_{\gamma,\delta}|z|^{-\frac{1/2}{n/2+\gamma}}d(z)^{\frac{\gamma-1/2}{n/2+\gamma}}\norm{W}_{L^{n/2+\gamma}},\quad z\in \C\setminus( \mathcal E_\delta\cup[0,\infty)). 
\end{align}
\end{lemma}

\begin{proof}
\eqref{lemma_EB_2_1} is a direct consequence of \eqref{theorem_KRS_1_1} and \eqref{corollary_KRS_1_1} with $1/p=1/2+1/(n+2\gamma)$ and $q=p'$. 

For the proof of \eqref{lemma_EB_2_2}, we take $\theta=\frac{2\gamma-1}{n+2\gamma}\in(0,1)$ so that
$
1-\theta=\frac{n+1}{n+2\gamma}
$. 
Interpolating between \eqref{theorem_KRS_1_1}  with $p=2(n+1)/(n+3)$, $q=p'$ and the trivial bound
$
\norm{R(z)}_{\mathbb B(L^2)}=\dist(z,[0,\infty))^{-1}
$ and, then, using H\"older's inequality, we obtain
\begin{align*}
\norm{W_1R(E)W_2}_{\mathbb B(L^2)}\le C_{\gamma,\delta} |z|^{-\frac{1-\theta}{n+1}}d(z)^{-\theta}\norm{W}_{L^{n/2+\gamma}}
=C_{\gamma,\delta}|z|^{-\frac{1/2}{n/2+\gamma}}d(z)^{\frac{\gamma-1/2}{n/2+\gamma}}\norm{W}_{L^{n/2+\gamma}}
\end{align*}
which completes the proof.
\end{proof}

\appendix

\section{Real interpolation and Lorentz space}
\label{appendix_interpolation}
Here a brief summery of real interpolation spaces and Lorentz spaces is given without proofs. One can find a much more detailed exposition in \cite{BeLo,Gra1}. 

A pair of Banach spaces $(\A,\B)$ is said to be a Banach couple if both $\A,\B$ are algebraically and topologically embedded in a Hausdorff topological vector space $\mathcal C$. Note that one can always take $\mathcal C$ to be a Banach space $\A_0+\A_1$. 
Given a Banach couple $(\A_0,\A_1)$ and $0<\theta<1$ and $1\le q\le \infty$, one can define a Banach space $\A_{\theta,q}=(\A_0,\A_1)_{\theta,q}$ by the so-called $K$-method, which satisfies that $(\A_0,\A_0)_{\theta,q}=\A_0$  and $(\A_0,\A_1)_{\theta,q}=(\A_1,\A_0)_{1-\theta,q}$ with equivalent norms and that if $1\le q_1\le q_2\le \infty$ then $(\A_0,\A_1)_{\theta,1}\hookrightarrow (\A_0,\A_1)_{\theta,q_1}\hookrightarrow (\A_0,\A_1)_{\theta,q_2}\hookrightarrow (\A_0,\A_1)_{\theta,\infty}$. Then the following real interpolation theorem is frequently used in this paper. 


\begin{theorem}[{\cite[Theorem 3.1.2]{BeLo},\cite{Pot}}]
\label{theorem_interpolation_2}
Let $(\A_0,\A_1)$ and $(\B_0,\B_1)$ be two Banach couples, $0<\theta<1$ and $1\le q\le\infty$. Suppose  that $T$ is a bounded linear operator from $(\A_0,\A_1)$ to $(\B_0,\B_1)$ in the sense  that $T:\A_j\to \B_j$ and
$\norm{T}_{\mathbb B(\A_j,\B_j)}\le M_j
$, $j=0,1$. Then $T$ is bounded from $\A_{\theta,q}$ to $\B_{\theta,q}$ and satisfies
$
\norm{T}_{\mathbb B(\A_{\theta,q},\B_{\theta,q})}\le M_0^{1-\theta}M_1^\theta
$. 
Moreover, if both $T:\A_0\to \B_0$ and $T:\A_1\to \B_1$ are compact, then $T:\A_{\theta,q} \to \B_{\theta,q}$ is also compact. 
\end{theorem}

Next we recall the definition and basic properties of Lorentz spaces. Given a $\mu$-measurable function $f$ on $\R^n$, we let $\mu_f(\alpha)=\mu(\{x\ |\ |f(x)|>\alpha\})$.
If we define the decreasing rearrangement of $f$ by $f^{*}(t) =\inf\{\alpha\ |\ \mu_f(\alpha)\le t\}$ then the Lorentz space $L^{p,q}(\R^n)$ is the set of measurable $f$ such that the following quasi-norm is finite: 
$$
\norm{f}^*_{L^{p,q}}:=\norm{t^{1/p-1/q}f^*(t)}_{L^q(\R_+,dt)}=p^{1/q}\norm{\alpha \mu_f(\alpha)^{1/p}}_{L^q(\R_+,\alpha^{-1}d\alpha)}<\infty
$$
Moreover, if $1<p<\infty$ and $1\le q\le \infty$ (which are sufficient for our purpose), then
$$
\norm{f}_{L^{p,q}}:=\norm{f^{**}}_{L^{p,q}}^*,\quad f^{**}(t):=\frac1t\int_0^tf^*(\alpha)d\alpha,
$$
becomes a norm on $L^{p,q}$ which makes $L^{p,q}$ a Banach space. Furthermore, $\norm{\cdot}_{L^{p,q}}$ is equivalent to $\norm{\cdot}^*_{L^{p,q}}$ in the sense that $\norm{f}^*_{L^{p,q}}\le \norm{f}_{L^{p,q}}\le C(p,q)\norm{f}^*_{L^{p,q}}$ with some constant $C(p,q)>0$. Thus all continuity estimates for linear operators can be expressed in terms of $\norm{\cdot}_{L^{p,q}}^*$. $L^{p,q}$ is increasing in $q$: $L^{p,1}\hookrightarrow L^{p,q_1}\hookrightarrow L^{p,p}=L^p\hookrightarrow L^{p,q_2}\hookrightarrow L^{p,\infty}$ if $1<q_1<p<q_2<\infty$. Moreover, $L^{p,q}$ is characterized by real interpolation: for $0<\theta<1$, $1<p_1<p_2<\infty$ with $\frac1p=\frac{1-\theta}{p_1}+\frac{\theta}{p_2}$ and $1\le q\le \infty$, one has $
(L^{p_0},L^{p_2})_{\theta,q}=L^{p,q}
$ with equivalent norms. If $1<p,q<\infty$ then $L^{p,q}(X;\C)'=L^{p',q'}(X;\C)$, where $r'=r/(r-1)$ is the H\"older conjugate of $r$. 

Finally we record two inequalities used frequently in this paper. First, for $1\le p,p_j<\infty$ and $1\le q,q_j\le \infty$ with $\frac1p=\frac{1}{p_1}+\frac{1}{p_2}$ and $\frac1q=\frac{1}{q_1}+\frac{1}{q_2}$, one has H\"older's inequality
\begin{equation}
\begin{aligned}
\label{Holder}
\norm{fg}_{L^{p,q}}&\le C\norm{f}_{L^{p_1,q_1}}\norm{g}_{L^{p_2,q_2}},\quad
\norm{fg}_{L^{p,q}}&\le C\norm{f}_{L^{\infty}}\norm{g}_{L^{p,q}}. 
\end{aligned}
\end{equation}
Secondly, for $1<s<n$, $1<p<q<\infty$, $\frac1p-\frac1q=\frac2n$ and $1\le r\le \infty$, we have the HLS inequality
\begin{align}
\label{HLS}
\norm{(-\Delta)^{-s/2}f}_{L^{q,r}}\le C\norm{f}_{L^{p,r}}. 
\end{align}




\begin{thebibliography}{99}

\bibitem{Agm}
S. Agmon, {\it Spectral properties of Schr\"odinger operators and scattering theory}, Ann. Scuola Norm. Sup. Pisa Cl. Sci. (4) \textbf{2} (1975), 151--218

\bibitem{Bec}
M. Beceanu, {\it New estimates for a time-dependent Schr\"odinger equation}, Duke Math. J. \textbf{159} (2011), 351--559

%
\bibitem{Bec2}M. Beceanu, {\it Structure of wave operators for a scaling-critical class of potentials}, Amer. J. Math. \textbf{136} (2014), 255--308

\bibitem{BeGo}
M. Beceanu, M. Goldberg, {\it Schr\"odinger dispersive estimates for a scaling-critical class of potentials}, Commun. Math. Phys. \textbf{314}  (2012), 471--481

\bibitem{BeKl}
M. Ben-Artzi, S. Klainerman, {\it Decay and regularity for the Schr\"odinger equation}, J. Ana lyse Math. \textbf{58} (2004) 25--37

\bibitem{BoMi}
J. -M. Bouclet, H. Mizutani, {\it Uniform resolvent and Strichartz estimates for Schr\"odinger equations with critical singularities}, Trans. Amer. Math. Soc. \textbf{370} (2018) 7293--7333

\bibitem{BPST2}
N. Burq, F. Planchon, J. G. Stalker, A.S. Tahvildar-Zadeh, {\it Strichartz estimates for the wave and Schr\"odinger equations with potentials of critical decay}, Indiana Univ. Math. J. \textbf{53}  (2004), 1665--1680

\bibitem{BVZ}
J. A. Barcel\'o, L. Vega, M. Zubeldia, {\it The forward problem for the electromagnetic Helmholtz equation with critical singularities}, Adv. Math. \textbf{240} (2013), 636--671

\bibitem{BeLo}
J. Bergh, J. L\"ofstr\"om, Interpolation spaces. An introduction, Springer- Verlag, Berlin, 1976, Grundlehren der Mathematischen Wissenschaften, No. 223.

\bibitem{COSY}
P. Chen, E. M. Ouhabaz, A. Sikora, L. Yan, {\it Restriction estimates, sharp spectral multipliers and endpoint estimates for Bochner-Riesz means}, J. Anal. Math. \textbf{129} (2016), 219--283

\bibitem{CoSi}
T. Coulhon and A. Sikora, {\it Gaussian heat kernel upper bounds via the Phragm\'en-Lindel\"of theo-
rem}, Proc. Lond. Math. Soc. (3) \textbf{96} (2008), 507--544

\bibitem
{Dan}P. D'Ancona, \emph{Kato smoothing and Strichartz estimates for wave equations with magnetic potentials}, Commun. Math. Phys. \textbf{335}  (2015), 1--16

\bibitem{Fos}
D. Foschi, {\it Inhomogeneous Strichartz estimates}, J. Hyperbolic Differ. Equ. \textbf{2} (2005), 1--24

\bibitem{Fra1}
R. L. Frank, {\it Eigenvalue bounds for Schr\"odinger operators with complex potentials}, Bull. Lond. Math. Soc. \textbf{43} (2011), 745--750

\bibitem{Fra2}
R. L. Frank, {\it Eigenvalue bounds for Schr\"odinger operators with complex potentials. III}, Trans. Amer. Math. Soc. \textbf{370} (2018), 219--240

\bibitem{GiVe}
J. Ginibre, G. Velo, {\it The global Cauchy problem for the non linear Schr\"{o}dinger equation}, 
Ann.\ lHP-Analyse\ non\ lin\'eaire.\ \textbf{2} (1985), 309--327

\bibitem{Gol}
M. Goldberg, {\it Strichartz estimates for the Schr\"odinger equation with time-periodic $L^{n/2}$ potentials},  J. Funct. Anal. \textbf{256} (2009), 718--746

\bibitem{GoSc}
M. Goldberg, W. Schlag, {\it A limiting absorption principle for the three-dimensional Schr\"odinger equation with $L^p$ potentials}, Int. Math. Res. Not., \textbf{75} (2004), 4049--4071

\bibitem{GVV}
M. Goldberg, L. Vega and N. Visciglia, {\it Counterexamples of Strichartz inequalities for
Schr\"odinger equations with repulsive potentials}, Int. Math. Res. Not. \textbf{2006} (2006) Article ID 13927

\bibitem{Gra1}
L. Grafakos, {\it Classical Fourier analysis. Second edition}, Graduate Texts in Mathematics, 249. Springer, New York, (2008)

\bibitem{Gut}
S. Guti\'errez, {\it Non trivial $L^q$ solutions to the Ginzburg-Landau equation}, Math. Ann. \textbf{328} (2004), 1--25

\bibitem{HiPh}E. Hille, R. S. Phillips, Functional analysis and semi-groups. Third printing of the revised edition of 1957. American Mathematical Society Colloquium Publications, Vol. XXXI. American Mathematical Society, Providence, R. I., 1974. xii+808 pp. 

\bibitem{Hor}
L. H\"ormander, {\it Estimates for translation invariant operators in $L^p$ spaces}, Acta Math. \textbf{104} (1960), 93--140

\bibitem{HYZ}
S. Huang, X. Yao, Q. Zheng, {\it $L^p$-limiting absorption principle of Schr\"odinger operators and applications to spectral multiplier theorems}, Forum Math. \textbf{30} (2018), 43--55

\bibitem{IoJe}
A. D. Ionescu, D. Jerison, {\it On the absence of positive eigenvalues of Schr\"odinger operators with rough potentials}, Geom. Funct. Anal. \textbf{13} (2003), 1029--1081


\bibitem{IoSc}
A. D. Ionescu, W. Schlag, {\it Agmon-Kato-Kuroda theorems for a large class of perturbations}, Duke Math. J. \textbf{131} (2006), 397--440


\bibitem{Jen_1}
A. Jensen, {\it Spectral properties of Schr\"odinger operators and time-decay of the wave functions}, Results in $L^2(\R^m)$, $m \ge 5$, Duke Math. J. \textbf{47} (1980), 57--80

\bibitem{Jen_2}
A. Jensen, {\it Spectral properties of Schr\"odinger operators and time-decay of the wave functions}, Results in $L^2(\R^4)$, J. Math. Anal. Appl. \textbf{101} (1984) 491--513

\bibitem{JeKa}
A. Jensen, T. Kato, {\it Spectral properties of Schr\"odinger operators and time-decay of the wave functions}, Duke Math. J. \textbf{46} (1979), 583--611

\bibitem
{Kat}
T. Kato, \emph{Wave operators and similarity for some non-self-adjoint operators}, Math. Ann. \textbf{162} (1965/1966), 258--279

\bibitem{Kat3}
T. Kato, {\it An $L^{q,r}$-theory for nonlinear Schr\"odinger equations}, Spectral and scattering theory and applications, Adv. Stud. Pure Math., vol. 23, Math. Soc. Japan, Tokyo, 1994, pp. 223--238

\bibitem{KaYa}
T. Kato, K. Yajima, \emph{Some examples of smooth operators and the associated smoothing effect}, Rev. Math. Phys. \textbf{1} (1989), 481--496

\bibitem{KeTa}
M. Keel, T. Tao, {\it Endpoint Strichartz estimates}, Amer. J. Math. \textbf{120} (1998), 955-980

\bibitem{KPV}
C. E. Kenig, G. Ponce, L. Vega, {\it Oscillatory integrals and regularity of dispersive equations}, Indiana Univ. Math. J. \textbf{40} (1991), 33--69

\bibitem{KRS}
C. E. Kenig, A. Ruiz, C. D. Sogge, {\it Uniform Sobolev inequalities and unique continuation for second order constant coefficient differential operators}, Duke Math. J. \textbf{55} (1987), 329--347

\bibitem{KMVZZ}
R. Killip, C. Miao, M. Visan, J. Zhang, J. Zheng, {\it Sobolev spaces adapted to the Schr\"odinger operator with inverse-square potential}, Math. Z. \textbf{288} (2018), 1273--1298

\bibitem{KoSe}
Y. Koh, I. Seo, {\it Inhomogeneous Strichartz estimates for Schr\"odinger's equation}, J. Math. Anal. Appl. \textbf{442} (2016), 715--725

\bibitem{Miz1}H. Mizutani, {\it Eigenvalue bounds for non-self-adjoint Schr\"odinger operators with the inverse-square potential}, to appear in J. Spectral Theory. http://arxiv.org/abs/1607.01727

\bibitem{Miz3}H. Mizutani, {\it Global-in-time smoothing effects for Schr\"odinger equations with inverse-square potentials}, Proc. Amer. Math. Soc. \textbf{146} (2018), 295--307

\bibitem{Pot}
F. Cobos, D. E. Edmunds, A. J. Potter, {\it Real interpolation and compact linear operators}, J. Funct. Anal. \textbf{88} (1990), 351--365

\bibitem{ReSi}
M. Reed, B. Simon, Methods of Modern Mathematical Physics III, IV, Academic Press, 1979, 1978

\bibitem{RoSc}
I. Rodnianski, W. Schlag, {\it Time decay for solutions of Schr\"odinger equations with rough and time-dependent potentials}, Invent. Math. \textbf{155}  (2004), 451--513

\bibitem{RoTa}
I. Rodnianski, T. Tao, {\it Effective limiting absorption principles, and applications}, Comm. Math. Phys. \textbf{333} (2015), 1--95 

%
\bibitem{Sim}B. Simon, {\it Schr\"odinger semigroups}, Bull. Amer. Math. Soc. (N.S.) \textbf{7} (1982) 447--526

\bibitem{Ste}
E. M. Stein, {\it Interpolation of linear operators}, Trans. Amer. Math. Soc. \textbf{83} (1956), 482--492

\bibitem{Ste2}
E. M. Stein, Singular Integrals and Differentiability Properties of Functions. Princeton Univ. Press, Princeton, NJ, 1970.

\bibitem{Str}
R. Strichartz, {\it Restrictions of Fourier transforms to quadratic surfaces and decay of solutions of wave equations}, Duke\ Math.\ J.\ \textbf{44} (1977), 705--714

%
\bibitem{Tao}
T. Tao, {\it Nonlinear Dispersive Equations: Local and Global Analysis}, CBMS Regional Series in Mathematics. Providence, RI: AMS, 2006

\bibitem{Tom}
P. Tomas, {\it A restriction theorem for the Fourier transform}, Bull. Amer. Math. Soc. \textbf{81} (1975), 477--478

\bibitem{Vil}
M. C. Vilela, {\it Inhomogeneous Strichartz estimates for the Schr\"odinger equation}, Trans.  Amer. Math. Soc. \textbf{359} (2007), 2123--2136

\bibitem{Yaj1}
K. Yajima, {\it The $W^{k,p}$-continuity of wave operators for Schr\"odinger operators}, J. Math. Soc. Japan \textbf{47} (1995), 551--581
\end{thebibliography}
\end{document}